\documentclass[compsoc,onecolumn,journal]{IEEEtran}
\RequirePackage[compress]{cite}
\input{a00-preamble.tex}

\newcommand{\citep}{\cite}
\newcommand{\citet}{\cite}

\usepackage{color,array,verbatim,algpseudocode,bm}
\definecolor{menucolor}{rgb}{0.1,0.52,0.47}
\definecolor{urlcolor}{rgb}{0.85,0.37,0.01}
\definecolor{runcolor}{rgb}{0.46,0.44,0.701}
\definecolor{filecolor}{rgb}{0.2,0.5,0.01}
\definecolor{linkcolor}{rgb}{0.12,0.47,0.70}
\definecolor{citecolor}{rgb}{0.55,0.36,0.01}
\definecolor{anchorcolor}{rgb}{0.4,0.4,0.4}
\usepackage[colorlinks=true, urlcolor=urlcolor, linkcolor=linkcolor,citecolor=citecolor,filecolor=filecolor, anchorcolor=anchorcolor,menucolor=menucolor]{hyperref}
\usepackage{nameref}
\usepackage{zref-xr}
\zxrsetup{
  tozreflabel=false,
  toltxlabel=true,
}

\usepackage[nodisplayskipstretch]{setspace}

\makeatletter
\newcommand\VeryLarge{\@setfontsize\Huge{16}{16}}
\newcommand\smaller{\@setfontsize\small{4}{4}}
\makeatother   

\def\subparagraph{} 
\usepackage{titlesec}
\titlespacing*{\section}{0pt}{*1}{*1}
\titlespacing{\subsection}{0pt}{*1}{*1}
\titlespacing{\subsubsection}{0pt}{*1}{*1}

\newtheorem{result}{Result}
\newtheorem{remark}{Remark}
\newtheorem{proposition}{Proposition}
\newtheorem{corollary}{Corollary}
\newtheorem{lemma}{Lemma}

\titleformat{\paragraph}[runin]{\normalfont\bfseries}{\theparagraph}{0.5em}{}[:]

\IEEEaftertitletext{\vspace{-\baselineskip}} 

\makeatletter
\renewenvironment{proof}[1][\proofname]{\par
  \pushQED{\qed}%
  \normalfont
  \topsep2pt \partopsep2pt 
  \trivlist
  \item[\hskip\labelsep
        \itshape
    #1\@addpunct{.}]\ignorespaces
}{%
  \popQED\endtrivlist\@endpefalse
  \addvspace{1pt plus 1pt} 
}
\makeatother

\newtheorem{theorem}{Claim}

\normalsize

\begin{document}
\title{
  \vspace{-0.2in}
  \VeryLarge The envelope of a complex Gaussian random variable}
\author{
  \vspace{-0.12in}
  Sattwik Ghosal\thanks{S.Ghosal is with the Department of Biostatistics, University of Michigan, Ann Arbor, Michigan, USA.} and Ranjan Maitra\thanks{R. Maitra is with the Department of Statistics, Iowa State University, Ames, Iowa, USA.}
\\ \vspace{-0.4in}
\thanks{A submission based on this manuscript won the first author a Student Poster award in Theoretical Statistics and Probability at the 2023 International Indian Statistical Association conference held at the Colorado School of Mines in Golden, CO, USA.}
\thanks{{\bf Abbreviations:} CAD, $\chi$-affine distribution; CDF, cumulative distribution function; IQC, identical quadratures components; MGF, moment generating function; MOM, method of moments; MRI, magnetic resonance imaging; PDF, probability density function; SNR, signal-to-noise ratio.}
}
\maketitle
\vspace{-0.2in}
\begin{abstract}
The envelope of an elliptical Gaussian complex vector, or equivalently, the amplitude or  norm of a bivariate normal random vector has application in many weather and signal processing contexts. We explicitly characterize its distribution in the general case through its probability density, cumulative distribution and moment generating function.  Moments and limiting distributions are also derived. These derivations are exploited to also characterize the special cases where the bivariate Gaussian mean vector and covariance matrix have a simpler structure, providing new additional insights in many cases. Simulations illustrate the benefits of using our formulae over Monte Carlo methods. We also use our derivations to get a better initial characterization of the distribution of the observed values in structural Magnetic Resonance Imaging datasets, and of wind speed.

\end{abstract}
\vspace{-0.1in}

\section{Introduction}
The envelope, amplitude, or norm, of a complex Gaussian
random variable has applications in engineering and scientific
disciplines such as meteorology~\citep{bestetal10,bailetal11,baggioandmuzy24}, signal processing, radar and other communications systems~\citep{beckmann67,beckmannandspizzichino63,simonetal95,proakis01,shietal05}, 
position localization~\citep{barshalometal01} and 
navigation~\citep{arafaandmessier10}, or nitude resonance~\citep{maitraandfaden09,maitra13} or diffusion weighted imaging~\citep{wegmannetal17}. It 
is defined to be $R{=}\sqrt{X_1^2{+}X_2^2}$, for the complex gain $X_1{+}\bi X_2$, where
$\bi{=}\sqrt{-1}$ and $(X_1,X_2)$ have a bivariate normal distribution
$\mN_2(\bmu,\bSigma)$ with
bivariate mean vector $\bmu$ and $2{\times}2$ dispersion matrix $\bSigma$.

Various cases of the envelope distribution have received special
names, definitions and treatments.  For instance, $R$
is said to have the Beckmann
distribution~\citep{beckmann67,beckmannandspizzichino63} when
$\bSigma$ is a diagonal matrix. This distribution reduces to the 
Hoyt/Nakagami-$q$ distribution~\citep{hoyt47,paris09,hajrietal09} when
additionally $\bmu{=}\bzero$. On the other hand, we get the Rice
distribution~\citep{rice44,rice45} when $\bmu{\neq}\bzero$ but
$\bSigma{\propto}\bI_2$, the identity 
matrix. For $\bmu{=}\bzero$ and $\bSigma{\propto}\bI_2$, $R$ has the
Rayleigh distribution~\citep{rayleigh1919}. In the most general
scenario, $R$ is a special case of the generalized Beckmann
distribution~\citep{zhuetal18} that is the distribution of the
Euclidean norm of a $p$-variate Gaussian random vector\footnote{Some
  authors call the generalized Beckmann distribution the generalized
  Rice distribution, but we feel the generalized Rice distribution is more ambiguous because a
Rice-distributed random variable is the envelope of a bivariate
Gaussian random vector with a spherical dispersion matrix. For greater
clarity, we feel that the generalized Beckmann distribution should be
used to refer to the distribution of the Euclidean norm of a
Gaussian random vector in its most general formulation.}. To fix
context, we call the $p{=}2$ case the second order generalized
Beckmann distribution. 

Complementing the envelope of a complex Gaussian random
variable is its phase that also has applications in signal
processing~\citep{pawulaetal82,pawula01,vitthaladevuniandalouini05}
and other
areas~\citep{jammalamadakaandsengupta01,mardiaandjupp00}. The  phase
distribution has received attention in the
communications literature~\citep{dharmawansaetal09}, and more extensively,
in
statistics~\citep{gattoandjammalamadaka07,gatto08,gatto09}, but the
same is not true for the envelope. Only the Rayleigh, Rice and Hoyt/Nakagami-$q$ distributions 
have been well-characterized or studied \citep{rice45,bozchalooiandliang07,hajrietal09,coluccia13,romeroetal15}, while only the expressions for the density~\citep{beckmann67} and the expected signal-to-noise ratio (SNR)~\citep{penaetal17}, but not the moments, have been derived for the Beckmann distribution. Similar characterizations do not exist for the second order generalized Beckmann distribution with general $\bSigma$, with its exact probability density function (PDF) as the only property that has been derived~\citep{aaloetal07,dharmawansaetal09}.
From an estimation perspective, the Beckmann is indistinguishable from the generalized Beckmann distribution or from the second order Identical Quadrature Components (IQC) model, which arises for the case when $\bSigma$ has identical diagonal elements. Nevertheless, there are situations~\citep{aaloetal07,dharmawansaetal09} where nonhomogeneous 
receiver quadrature error or I-Q gain mismatch, and correlated
Gaussian noise~\citep{parkandcho05} can yield envelopes from this
distribution, and  it is of importance to characterize its
properties. Therefore, in this paper, we explicitly derive, in 
Section~\ref{sec:main}, the cumulative distribution function (CDF) of
the second order generalized Beckmann distribution. For general $p$,
\citet{zhuetal18} provided upper and lower bounds for the CDF of the
generalized Beckmann distribution, but we provide explicit
representations of the CDF. We also use this opportunity to lay out a 
detailed derivation of the PDF to supplement the sketch provided of a
similar specification~\citep{dharmawansaetal09}, and that itself is of
an alternative form to the one in~\citep{aaloetal07}. We next provide the
moment generating function~(MGF) $\M_R(t)$ of the second order generalized
Beckmann distribution, after showing that it exists for any finite
$t\in\R$. Formulae for the raw moments, and limiting distributions for the generalized Beckmann distribution are also provided. 
The derived formulae for the PDF, CDF,  MGF, and the moments are then
applied in Section~\ref{sec:special} in the case of the
specialized envelope distributions, namely, the Rayleigh,  Rice,
Beckmann and Hoyt/Nakagami-$q$ distributions. In some cases, our
formulae match existing formulae obtained through other means, while
in other cases, our methods provide properties of these distributions
not hitherto derived. We also provide a deeper study of the second order IQC distribution. 
Section~\ref{sec:simulations} evaluates performance of our moments formulae vis-a-vis Monte Carlo methods, both in terms of computational speed and accuracy.  We see that our exact formulae are generally accurate at a small fraction of computational cost, but there are situations where numerical issues arise in calculating our exact formulae and then formulae obtained from the limiting distributions are a better option. Section~\ref{sec:applications} uses our derivations in Sections~\ref{sec:main} and \ref{sec:special} to obtain Method-of-Moments (MOM) estimators for the IQC model parameters for magnitude Magnetic Resonance Imaging (MRI) and wind speed data. Our analysis provides evidence of the IQC model being a better fit to the data than its Rice counterpart. We conclude with some discussion. Our article also has appendices containing some needed and, in some cases, more general technical details. 

\section{Main Results}
\label{sec:main}
\subsection{Background and Preliminaries}
\label{sec:fns}
The PDF of the second order generalized Beckmann distribution has been
specified in two alternative ways~\citep{aaloetal07,dharmawansaetal09}. Our
version is similar to that in~\citet{dharmawansaetal09} who only
provided a very terse sketch, so we use this opportunity to rewrite
the PDF and provide a formal proof for a fuller reference. 
\begin{result}
  \label{theo:2d.PDF}
Let $\bX {=} (X_1,X_2){\sim} \mN_2(\bmu,\bSigma)$, where $\bmu {=}
(\mu_1,\mu_2)^\top$, and $\bSigma$ has diagonal elements $\sigma_1^2$
and $\sigma_2^2$ and off-diagonal element $\rho\sigma_1\sigma_2$. The
PDF of $R{=}\sqrt{X_1^2{+}X_2^2}$ is
\begin{equation}
  \label{eq:2d.PDF}
  \begin{split}
    f\!_R&(r;\bmu, \bSigma){=}\alpha r \exp{({-}\beta r^2)}\Big\{\!\!\sum_{j=0}^{\infty} \!\epsilon_j\I_{2j}(\psi r )\I_{j}(\eta r^2)\cos
    2j \delta\Big\}\1[r{>}0],
\end{split}
    \end{equation}
  where $\1(\cdot)$ is the indicator
  function, $\epsilon_j{=}2^{\1[j>0]}$, $\I_{m}(.)$ is the modified
  Bessel function of the first kind
  of the $m$th order,  $\delta {=}
  (\tilde{\phi}{-}\phi/2)$, where $\tilde\phi$ and $\phi$ satisfy
  \begin{align*}
	  \cos\tilde\phi & = \frac{\mu_2\sigma^2_1{-}\rho\mu_1\sigma_2\sigma_1}{\sqrt{\mu_1^2\sigma_2^4 +\mu_2^2\sigma_1^4+\rho^2\sigma_1^2\sigma_2^2(\mu_1^2+\mu_2^2)-2\rho\mu_1\mu_2\sigma_1\sigma_2(\sigma_1^2+\sigma_2^2)}},\\
	\sin\tilde\phi & = \frac{\mu_1\sigma^2_2{-}\rho\mu_2\sigma_2\sigma_1}{\sqrt{\mu_1^2\sigma_2^4 +\mu_2^2\sigma_1^4+\rho^2\sigma_1^2\sigma_2^2(\mu_1^2+\mu_2^2)-2\rho\mu_1\mu_2\sigma_1\sigma_2(\sigma_1^2+\sigma_2^2)}},\\
	\cos\phi & = \frac{\sigma_1^2 -\sigma_2^2}{\sqrt{(\sigma_1^2 -\sigma_2^2)^2 + 4\rho^2\sigma_1^2\sigma_2^2}},\qquad 	\sin\phi =  \frac{2\rho\sigma_1\sigma_2}{\sqrt{(\sigma_1^2 -\sigma_2^2)^2 + 4\rho^2\sigma_1^2\sigma_2^2}},
\end{align*}
and
\begin{align*}
    \alpha & {=} \frac1{\sigma_1\sigma_2\sqrt{1{-}\rho^2}}
             \exp{\left({-}\frac{\mu^2_1\sigma_2^2
             {+} \sigma_1^2\mu_2^2 - 2\rho\sigma_1\sigma_2\mu_1\mu_2}{2
             \sigma^2_1\sigma_2^2(1{-}\rho^2)}\right)}>0, \qquad
    \beta   {=} \frac{\sigma_1^2 {+}
             \sigma_2^2}{4
             \sigma^2_1\sigma_2^2(1{-}\rho^2)}
    > 0, \\
    \eta &
    {=}\frac{\sqrt{(\sigma^2_2{-}\sigma^2_1)^2{+}4\sigma^2_1\sigma^2_2\rho^2}}{4\sigma^2_1\sigma^2_2(1{-}\rho^2)}
      \geq 0, \qquad
 \psi {=}
    \frac{\sqrt{\sigma^2_1(\rho\mu_1\sigma_2{-}\mu_2\sigma_1)^2{+}\sigma^2_2(\rho\mu_2\sigma_1{-}\mu_1\sigma_2)^2}}{\sigma^2_1\sigma^2_2(1{-}\rho^2)}
         \geq 0,
  \end{align*}
\end{result}
\begin{proof}
  See Appendix~\ref{proofPDF}.
\end{proof}
\begin{remark}
  D. Paindaveine has pointed out a potential alternative indirect way
  to finding the PDF (and the 
  CDF) of a generalized Beckman distribution. We can write the squared
  Euclidean norm of a complex Gaussian random variable as a scaled
  sum of independent non-central chi-squared random variables, and
  then use Theorem 4.2b.1 of 
  \citet{mathaiandprovost92} to obtain the PDF and the CDF of the squared
  Euclidean norm, and from there, the PDF and CDF of the
  envelope. However, Theorem 4.2b.1 of 
  \citet{mathaiandprovost92} represents the PDF and the CDF as an
  alternating series, which, because of their alternating signs, are
  plagued by slow convergence~\citep{lange10}. Our direct derivations
  here are somewhat less concerning because the only term involving potentially involving a negative sign is the term involving the cosine, but from 8.334 of~\citet{gradshteynandryzhik00}, we know that 
  $$\cos 2j\delta = \frac\pi{\Gamma\left(\frac12 + \frac{2j\delta}{\pi}\right) \Gamma\left(\frac12 - \frac{2j\delta}{\pi}\right)},$$
  but the gamma function $\Gamma(x)$ is negative only when $x\in (-2k-1,-2k)$ for any  $k\in\mathbb Z$. This means that $\cos2j\delta$ has negative contributions for a given $\delta$ when $\frac{(4k+1)\pi}{4\delta}{<} j {<} \frac{(4k+3)\pi}{4\delta}$, and depending on the value of $\delta$, may not alternate for sets of successive $j$s.
\end{remark}
\subsection{Additional characterization of the second order Beckman distribution}

\subsubsection{The cumulative distribution function}
\label{sec:cdf}
We now provide explicit forms of the CDF of a generalized Beckmann
random variable of second order.
\begin{theorem}
  \label{theo:2d.CDF} Let 
  $\ell_{k_1,k_2}^{(j)}{\doteq}k_1{+}2k_2{+}2j$, 
  while $(n)_{(l)}{\doteq}
  n!/(n{-}l)!$ denotes
  a falling factorial~\citep[Page 48]{grahametal94}, and 
  \begin{align*}
    C(k_1,k_2,j;\psi,\eta) & {=} \frac{\psi^{2k_1}\eta^{2k_2}}{4^{k_1+k_2}k_1!(k_1+2j)!k_2!(k_2+j)!},\\    
    T_1(k_1,k_2,j;\psi,\eta)&
    {=}C(k_1,k_2,j;\psi,\eta)
          \frac{\ell_{k_1,k_2}^{(j)}!}{\beta^{\ell_{k_1,k_2}^{(j)}}},
          \mbox{ and }\\
    T_2(u,k_1,k_2,j,k;\psi,\eta)& {=}
C(k_1,k_2,j;\psi,\eta) \frac{(\ell_{k_1,k_2}^{(j)})_{(k{-}1)}}{\beta^{k{-}1}}u^{2\ell_{k_1,k_2}^{(j)}{+}2{-}2k}.
  \end{align*}
  Under the definitions and setting of  Result~\ref{theo:2d.PDF},
  $R$ has  CDF
  \begin{equation}
    \begin{split}
      F_{R}(u;\bmu,\bSigma)       {=}\frac{\alpha}{2\beta}
\sum_{j{=}0}^{\infty}\epsilon_j\!\!
       \left(\frac{\eta\psi^2}8\right)^j&\!\!\!\cos2j\delta
       \Bigg\{\sum_{k_1{=}0}^{\infty}\sum_{k_2{=}0}^{\infty}
                             T_1(k_1,k_2,j;\psi,\eta){-}\exp{({-}\beta u^2)}\sum_{k_1{=}0}^{\infty}\!\sum_{k_2{=}0}^{\infty}\!\!\!\!\sum_{k{=}1}^{\ell_{{k_1,k_2}}^{(j)}{+}1}\!\!\!\!\!T_2(u,k_1,k_2,j,k;\psi,\eta)\Bigg\},\\
    \end{split}
    \label{eq:CDF}
  \end{equation}
for $u > 0$, and is zero for $u{\leq} 0$.
\end{theorem}
\begin{proof} By definition,
$
F_{R}(u;\bmu,\bSigma){=}\int_{0}^{u}\!\!\! f_{R}(r;\bmu,\bSigma)dr$, where
$f_{R}(r;\bmu,\bSigma)$ is as in~\eqref{eq:2d.PDF}.
Define $${\Delta}^{u}_{n}=\int_{0}^{u}x^n \exp{({-}\beta x^2)}dx.$$
Integrating by parts yields the recursive relation  
\begin{align*}
\Delta^{u}_n{=}\int_{0}^{u}\!\!\!x^n
  \exp{({{-}\beta}x^2)}dx
                          {=}{-}\frac{u^{n{-}1}\exp{({-}\beta  u^2)}}{2\beta}{+}\frac{n{-}1}{2\beta}\Delta^{u}_{n{-}2}.
\end{align*}
Then, with $i_{[k]} {\doteq} i(i {-} 2)(i {-} 4)\hdots(i {-} 2k {+}  2)$, for even $n$, 
\begin{equation} 
  \begin{split}
  \Delta^{u}_n=&{-}\frac{1}{2\beta}\exp{({-}{\beta}u^2)}\sum_{k{=}1}^{\frac{n}{2}}\frac{u^{n{+}1{-}2k}}{(2\beta)^{k{-}1}}(n{-}1)_{[k{-}1]}{+}\frac{(n{-}1)_{[\frac{n}{2}{-}1]}}{(2\beta)^{\frac{n}{2}}}\sqrt{\frac{\pi}{\beta}}[\Phi(u\sqrt{2\beta}){-}1/2],\\
=&-\frac{1}{2\beta}\exp{({-}{\beta}u^2)}\sum_{k=1}^{\frac{n}{2}}\frac{u^{n+1-2k}}{(2\beta)^{k-1}}(n-1)_{[k-1]}
  {+}\frac{\Gamma\left(\frac{n+1}{2}\right)}{2\beta^{\frac{n+1}{2}}}[\Phi(u\sqrt{2\beta})-1/2]
  \end{split} \label{eq:delta.even}
  \end{equation}
while,
for odd $n$, 
\begin{align}
  \label{eq:delta.odd}
    \Delta^{u}_n=&\frac{\left(\frac{n{-}1}{2}\right)!}{2\beta^{\frac{n{+}1}{2}}}{-}\frac{1}{2\beta}\exp{({-}\beta
    u^2)}\sum_{k{=}1}^{\frac{n{+}1}{2}}\frac{u^{n{+}1{-}2k}}{\beta^{k{-}1}}\left(\frac{n{-}1}{2}\right)_{(k{-}1)}   =\frac{\Gamma\left(\frac{n{+}1}{2}\right)}{2\beta^{\frac{n{+}1}{2}}}{-}\frac{1}{2\beta}\exp{({-}\beta
    u^2)}\sum_{k{=}1}^{\frac{n{+}1}{2}}\frac{u^{n{+}1{-}2k}}{\beta^{k{-}1}}\left(\frac{n{-}1}{2}\right)_{(k{-}1)} 
\end{align}
 For any integer
$j{\geq}0$,
$\I_{j}(z){=}\left(\frac{z}{2}\right)^j\sum_{k=0}^{\infty}\frac{z^{2k}}{4^{k}k!\,\Gamma(k+j+1)}.$
So
\begin{equation}
  \begin{split}
\I_{2j}(\psi r)\I_{j}(\eta r^2)
&=\left(\frac{1}{8}\eta\psi^2 r^4\right)^j\sum_{k_1=0}^{\infty}\sum_{k_2=0}^{\infty}\frac{(\frac{1}{4}\psi^2 r^2)^{k_1}(\frac{1}{4}\eta^2 r^4)^{k_2}}{k_1!\,k_2!\,\Gamma(k_1{+}2j{+}1)\Gamma(k_2{+}j{+}1)}
                =\left(\frac{1}{8}\eta\psi^2
               \right)^j\sum_{k_1=0}^{\infty}\sum_{k_2=0}^{\infty}
                  C(k_1,k_2,j;\psi,\eta) r^{2k_1+4k_2+4j}\\
  \end{split}
  \label{eq:prod.Bessel}
\end{equation}
which, upon combining with \eqref{eq:delta.odd}, yields
\begin{equation}
\begin{split}
\int_{0}^{u}&\!\!\alpha r \exp{(-\beta r^2)}\I_{2j}(\psi r)\I_{j}(\eta r^2)dr\\
  &{=} \sum_{k_1=0}^{\infty}\sum_{k_2=0}^{\infty}\alpha
     \left(\frac{1}{8}\eta\psi^2 \right)^j C(k_1,k_2,j;\psi,\eta)
    \Delta^{u}_{2\ell_{k_1,k_2}^{(j)}+1}\\ &{=}\frac{\alpha}{2\beta}\left(\frac{1}{8}\eta\psi^2
  \right)^j \Big\{\sum_{k_1=0}^{\infty}\sum_{k_2=0}^{\infty}
  T_1(k_1,k_2,j;\psi,\eta)
  {-}\exp{(-\beta
  u^2)}\sum_{k_1=0}^{\infty}\sum_{k_2=0}^{\infty}\!\!\!\!\sum_{k{=}1}^{\ell_{{k_1,k_2}}^{(j)}{+}1}\!\!\!\!\!T_2(u,k_1,k_2,j,k;\psi,\eta)\Big\},
\end{split}
\label{eq:insum}
\end{equation}
from where we get \eqref{eq:CDF}, after multiplying each term with
$\epsilon_j\cos 2j\delta$ and summing over $j\in\{0,1,2\hdots\}$. 
\end{proof}

\begin{proposition}
  \label{theo:altCDF}
Under the framework of Theorem~\ref{theo:2d.CDF}, the CDF of $R$ is equivalently,
\begin{equation}
  \begin{split}
    F_{R}(u;\bmu,\bSigma)
           {=}\frac{\alpha}{2\beta}\sum_{j=0}^{\infty}
           \epsilon_{j}cos2j\delta\left(\frac{\eta
           \psi^2}{8}\right)^{j}
           \sum_{k_1=0}^{\infty}\!\sum_{k_2=0}^{\infty}&\!C(k_1,k_2,j,\psi,\eta)  {\frac{ u^{2\ell^{(j)}_{k_1,k_2}{+}2}}{\ell^{(j)}_{k_1,k_2}{+}1}{}
    _{1}\F_{1}\left(\ell^{(j)}_{k_1,k_2}{+}1,\ell^{(j)}_{k_1,k_2}{+}2,-\beta u^2\right)}
  \end{split}
    \label{eq:alt.CDF}
  \end{equation}
for $u > 0$ and zero everywhere else. Here, ${}_1\F_1(\cdot,\cdot,\cdot)$ is the confluent hypergeometric
function, or Kummer's function, of the first kind~\citep{kummer1837,abramowitzandstegun64}.
\end{proposition}
\begin{proof}
	Note $\Delta_n^u{=} \frac{\gamma\left(\frac{n+1}2,\beta u^2\right)}{2\beta^{\frac{n+1}2}},$ with
$\gamma(a,x){=}\int_0^xt^{a-1}\exp{(-t)}dt$, the lower incomplete
gamma function. Also, from (13.6.10) of~\citet{abramowitzandstegun64},
$\gamma(a,x){=} a^{-1}x^a{}_1\F_1(a,a+1,-x)$
. Then, 
$$
\Delta^{u}_{2\ell_{k_1,k_2}^{(j)}{+}1} = \frac{ u^{2l^{(j)}_{k_1,k_2}{+}2}}{2(\ell^{(j)}_{k_1,k_2}{+}1)}{}_{1}\F_{1}\left(\ell^{(j)}_{k_1,k_2}{+}1,\ell^{(j)}_{k_1,k_2}{+}2,-\beta u^2\right).
$$
The first line in the right hand side of \eqref{eq:insum} is also
expressed as
\begin{equation}
  \begin{split}
\int_{0}^{u}&\!\!\alpha r \exp{(-\beta r^2)}\I_{2j}(\psi
                                             r)\I_{j}(\eta r^2)dr 
    {=} \frac{\alpha}{2\beta}\left(\!\frac{\eta
    \psi^2}{8}\!\right)^{j}\!\!\sum_{k_1=0}^{\infty}\sum_{k_2=0}^{\infty}
    C(k_1,k_2,j,\psi,\eta)\frac{
      u^{2l^{(j)}_{k_1,k_2}{+}2}}{\ell^{(j)}_{k_1,k_2}{+}1}{} _{1}\F_{1}\left(\ell^{(j)}_{k_1,k_2}{+}1,\ell^{(j)}_{k_1,k_2}{+}2,-\beta  u^2\right),
    \end{split}
\end{equation}
from where we get~\eqref{eq:alt.CDF}, in the same manner as~\eqref{eq:CDF} is
obtained from~\eqref{eq:alt.CDF}.
\end{proof}

\begin{remark}
Theorem~\ref{theo:2d.CDF} and Proposition~\ref{theo:altCDF} provide
two alternative versions of the CDF. In general, \eqref{eq:CDF}
involves fewer terms to calculate in the third series, because of 
reductions obtained by analytical integration of
$\Delta_{2\ell_{k_1,k_2}^{(j)}{+}1}^u$. However, as seen, for example, in
Section~\ref{sec:Rice} or in Section~\ref{sec:back}, there are some 
special cases where Proposition~\ref{theo:altCDF} may provide faster
calculations because of the direct calculation of~\eqref{eq:alt.CDF}
through high-precision numerical algorithms in standard software libraries.
\end{remark}

\subsubsection{The Moment Generating Function}
\label{sec:mgfs}
\begin{theorem}\label{theo:mgf}

  Under the framework and definitions of Result~\ref{theo:2d.PDF} and
  Theorem~\ref{theo:2d.CDF}. the MGF $\M_R(t)$, of $R$ exists
  $\forall t\in\R$, and is
  \begin{equation}
    \begin{split}
    \M_R(t) {=}&\alpha\sum_{j=0}^{\infty}\!\epsilon_j\cos 2j
    \delta\!\left(\frac{1}{8}\eta\psi^2 \right)^j\!\!\!\sum_{k_1=0}^{\infty}\sum_{k_2=0}^{\infty}
                 C(k_1,k_2,j;\psi,\eta)       \mI^{(\beta)}_{2\ell_{k_1,k_2}^{(j)}{+}1}(t),
    \end{split}
        \label{eq:mgf}
      \end{equation}
  where, for any odd integer $m$, we define 
  \begin{equation}
    \begin{split}
      \mI_m^{(\beta)}(t){=}& \frac{\Gamma\left(\frac{m{+}1}2\right)}{2\beta^{\frac{m{+}1}{2}}}
                       {}_{1}\F_{1}\left(\frac{m{+}1}2,\frac{1}{2},\frac{t^2}{4\beta}\right){+}t\,\frac{\Gamma\left(\frac{m}2{+}1\right)}{2\beta^{m/2+1}}\,
        {}_{1}\F_{1}\left(\frac{m}{2}{+}1,\frac{3}{2},\frac{t^2}{4\beta}\right),
    \end{split}
    \label{eq:int}
  \end{equation}
  which admits an alternative representation given by
   \begin{equation}
    \begin{split}
      \mI_m^{(\beta)}(t)&{=}
      \frac{\sqrt{\pi}\exp{\left(\frac{t^2}{4\beta}\right)}}{2\beta^{\frac{m{+}1}2}}\left\{
      \Gamma\left(\frac m2{+}1\right) \Lag_{\frac
      m2}^{({-}\frac12)}\Bigg({-}\frac{t^2}{4\beta}\right){+}\frac{t}{2\sqrt{\beta}}\Gamma\left(\frac {m{+}1}2\right) \Lag_{\frac
      {m{-}1}2}^{(1/2)}\left({-}\frac{t^2}{4\beta}\right)\Bigg\}.
    \end{split}
    \label{eq:alt.I}
  \end{equation}
Here, $\Lag_\nu^{(a)}(x)$ is a Laguerre function, as introduced
by \citet{pinney46}, for unrestricted $\nu$. When  $\nu$ is a
nonnegative integer, $\Lag_\nu^{(a)}(
x)$ is more commonly known as an
{\em associated Laguerre polynomial}.   
\end{theorem}
\begin{proof} We have
  $\M_R(t){=}\E\{\exp{(tR)}\}{=}\int_{0}^{\infty}\!\!\!\exp{(tr)}f_{R}(r)dr$,
  with $f_R(r)$ as in~\eqref{eq:2d.PDF}. 
Therefore, 
  \begin{equation}\M_R(t){=}\alpha\sum_{j=1}^{\infty}\!\!\epsilon_j\cos 2j
    \delta\!\!\int_0^\infty\!\!\!\!\!\!r\exp{(tr{-}\beta
    r^2)}\{\I_{2j}(\psi r)\I_{j}(\eta r^2)\}dr.
  \label{eq:MGF}
\end{equation}
Existence of a MGF is almost immediate. For we have that $\I_{j}(\kappa){\leq}
\I_{0}(\kappa),\mbox{ } \forall j\geq 1$. Also $\abs{\cos
  2j\delta}<1$, and the integrands are all nonnegative. Hence,
\begin{align*}
  \abs{\M_R(t)} & \leq\alpha\!\!\int_0^\infty\!\!\!\!\!\!r\exp{(tr{-}\beta                  r^2)}\I_{0}(\psi r)\sum_{j=1}^{\infty}\!\!\epsilon_j\I_{j}(\eta r^2)dr.\\
 \end{align*}
 Also 
 $\I_0(\psi r){<}\exp{(\psi r)}$, 
$\sum_{j=1}^{\infty}\epsilon_j\I_{j}(\eta r^2) {=} \exp{(\eta
  r^2)}$, and $\beta {\geq} \eta$ since $\abs{\rho}{\leq} 1.$
Therefore,
\begin{equation}
  \begin{split}
\abs{\M_R(t)} & {\leq}  \int_{0}^{\infty}\!\!\!\!\!\alpha r \exp{(tr{-}\beta
  r^2{+}\psi r{+}\eta r^2)}dr = \alpha
\exp{\left\{\frac{(t{+}\psi)^2}{4(\beta{-}\eta)}\right\}}\!\!\int_{0}^{\infty}\!\!\!\!\!\!r
\exp{[{-}\{r{-}\frac{t{+}\psi}{2\sqrt{\beta{-}\eta}}\}^2]}dr
    {<}\infty
    \end{split}
\end{equation}
for $t\in\R$, since
the integral is proportional to $\E({W\1[W\geq0]})$, with $W{\sim }
\mN(\frac{t+\psi}{2\sqrt{\beta-\eta}}, \frac12)$.

It remains to
show~\eqref{eq:mgf}. 
  From~\eqref{eq:prod.Bessel}, 
\begin{align*}
    \int_{0}^{\infty}\!\!\!\!\!r &\exp{(tr{-}\beta r^2)}\I_{2j}(\psi r)\I_{j}(\eta r^2) dr
{=} \left(\frac{1}{8}\eta\psi^2 \right)^j\sum_{k_1=0}^{\infty}\sum_{k_2=0}^{\infty}
                                                                                                C(k_1,k_2,j;\psi,\eta)\mI_{2\ell_{k_1,k_2}^{(j)}{+}1}(t;\beta)
\end{align*}
where
\begin{equation}
 \mI^{(\beta)}_m(t) {=} \int_{0}^{\infty}\!\!\!r^m\exp{(tr{-}\beta
  r^2)}dr {=}
\sum_{k=0}^{\infty}\frac{t^{k}}{k!}{\Delta}^\infty_{m+k},
\label{eq:I}
\end{equation}
and $m$ is odd. Splitting the series in terms of series of odd and even terms, and
using $\Delta^{\infty}_n{=}\frac{\Gamma\left(\frac{n{+}1}{2}\right)}{2\beta^{\frac{n{+}1}{2}}}$
for even $n$, from \eqref{eq:delta.even}, and 
for odd $n$ from \eqref{eq:delta.odd}, we get
\begin{equation}
\begin{split}
  \mI^{(\beta)}_m(t) &{=}
  \sum_{k=0}^{\infty}\frac{t^{2k}}{(2k)!}\frac{(\frac{m-1}2{+}k)!}{2\beta^{\frac{m+1}2{+}k}}
                         {+}\sqrt{\pi}\sum_{k=0}^{\infty}\!\!\frac{t^{2k{+}1}}{(2k{+}1)!}\frac{(m{+}2k{+}1)!}{(\frac{m{-}1}2{+}k{+}1)!}\frac1{2^{m+2k+2}\beta^{\frac{m}2{+}k{+}1}}\\
  &= \frac{\left(\frac{m{-}1}2\right)!}{2\beta^{\frac{m{+}1}{2}}}
    {}_{1}\F_{1}\left(\frac{m{+}1}2,\frac{1}{2},\frac{t^2}{4\beta}\right)
   {+}t\frac{\sqrt\pi}{(4\beta)^{m/2+1}}\frac{(m+1)!}{\left(\frac{m{+}1}2\right)!}     {}_{1}\F_{1}\left(\frac{m}{2}{+}1,\frac{3}{2},\frac{t^2}{4\beta}\right)\\
  &= \frac{\Gamma\left(\frac{m{+}1}2\right)}{2\beta^{\frac{m{+}1}{2}}}
    {}_{1}\F_{1}\left(\frac{m{+}1}2,\frac{1}{2},\frac{t^2}{4\beta}\right){+}t\,\frac{\Gamma\left(\frac{m}2{+}1\right)}{2\beta^{m/2+1}}     {}_{1}\F_{1}\left(\frac{m}{2}{+}1,\frac{3}{2},\frac{t^2}{4\beta}\right)
    ,
\end{split}
\label{eq:Is}
  \end{equation}
from where~\eqref{eq:mgf} follows.

To obtain the alternative representation in~\eqref{eq:alt.I}, we note
that from (13.1.27) of~\citet{abramowitzandstegun64}, we get the relationships:
\begin{equation}
\begin{split}
  {}_1\F_1\left(\frac {m+1}2,\frac12,\frac{t^2}{4\beta}\right) & =
                                                     \exp{\left(\frac{t^2}{4\beta}\right)}{}_1\F_1\left(-\frac{m}2,\frac12,-\frac{t^2}{4\beta}\right),
 \\
 {}_1\F_1\left(\frac
  m2{+}1,\frac32,\frac{t^2}{4\beta}\right) & =
                                             \exp{\left(\frac{t^2}{4\beta}\right)}{}_1\F_1\left(-\frac{m-1}2,\frac32,-\frac{t^2}{4\beta}\right).\\
  \end{split}
\label{eq:I.1F1a}
\end{equation}
From \S 16.1 of~\citet{whittakerandwatson1915}, and then (2.8) of \citet{pinney46}
\begin{align*}
  _1\F_1(a,b,x) & =  \exp{\left(\frac x2\right)}x^{-\frac b2}\mM_{\frac
                 b2-a,\frac{b-1}2}(x)= \frac{\Gamma(b)\Gamma(1-a)}{\Gamma(b-a)} \Lag_{-a}^{(b-1)}(x),
\end{align*}
where $\mM_{\kappa,\mu}(x)$ is one of the two Whittaker functions~\citep{whittaker1903}. Therefore,
\begin{equation}
  \begin{split}
  {}_1\F_1\left(-\frac{m}2,\frac12,-\frac{t^2}{4\beta}\right)  & =
  \frac{\Gamma\left(\frac12\right)\Gamma\left(\frac
      m2{+}1\right)}{\Gamma\left(\frac{m+1}2\right)}\Lag_{\frac{m}2}^{(-1/2)}\left(-\frac{t^2}{4\beta}\right),\\
  {}_1\F_1\left(-\frac{m{-}1}2,\frac32,-\frac{t^2}{4\beta}\right) & =
  \frac{\Gamma\left(\frac32\right)\Gamma\left(\frac{m+1}2\right)}{\Gamma\left(\frac
                                                                 m2+1\right)}\Lag_{\frac{m{-}1}2}^{(1/2)}\left(-\frac{t^2}{4\beta}\right). \\
    \end{split}
  \label{eq:I.1F1b}
\end{equation}
Combining~\eqref{eq:I.1F1a} and \eqref{eq:I.1F1b} and inserting into
\eqref{eq:Is} yields \eqref{eq:alt.I}.
\end{proof}

\begin{remark}
  We make two comments on our results.
  \begin{enumerate}
  \item Our Laguerre functions are either associated Laguerre
    polynomials of nonnegative integer order (with $\alpha{=}\frac12$),
    or Laguerre functions     of half-integer order (with
    $\alpha{=-}\frac12$), for which analytical expressions can aid
    computations. 
  \item 
    Appendix~\ref{sec:alterMGF} derives an alternative form of the MGF of
the second order GBD. However, while that form does not require the
evaluation of first order confluent hypergeometric functions, it requires the
calculation of more terms, and so we consider Theorem~\ref{theo:mgf} to be the
preferred approach, with the Laguerre functions or
$_1\F_1(\cdot,\cdot,\cdot)$ evaluated in high precision using
standard software libraries.  
\end{enumerate}
\end{remark}

\subsubsection{Moments}
\label{sec:moments}
Even-ordered raw moments, especially of low order, can be easily obtained from the definition of $R$, but a formula for the general $s$th moment is found as a corollary to Theorem~\ref{theo:mgf} as follows:
\begin{corollary}
  Under the framework and definitions of Result~\ref{theo:2d.PDF} and
  Theorems~\ref{theo:2d.CDF} and \ref{theo:mgf}, the $s$th raw moment
  of $R$ is $\upmu_s {=} \E(R^s)$, and given by the formula
  \begin{equation}
	    \upmu_s{=}\frac{\alpha}{2\beta^{\frac s2+1}}\sum_{j=0}^{\infty}\epsilon_j\cos 2j\delta\left(\frac{\psi^2}{4\beta}\right)^{j}   \sum_{k{=}0}^{\infty}
    \left(\frac\eta{2\beta}\right)^{2k+j}
    \frac{\Gamma\left(\frac{s}{2}{+}2k{+}2j{+}1\right)}{k!(k+j)!} 
											{    {{}_{1}\F_{1}\left(\frac{s}{2}{{+}}2k{+}2j{{+}}1,2j{+}1,\frac{\psi^2}{4\beta}\right).}}
      \label{eq:mus}
  \end{equation}
  \label{theo:mus}
\end{corollary}
\begin{proof}
We obtain the $s$th raw moment from the MGF from $\upmu_s {=} \frac
{d^s}{dt^s} \M_R(t)\Big|_{t=0}.$ The only part of $\M_R(t)$ involving
$t$ in~\eqref{eq:mgf} is
$\mI^{(\beta)}_{2\ell_{k_1,k_2}^{(j)}{+}1}(t)$.
Let $\lceil w \rceil$ be the smallest integer not below $w$. 
From the second equality in \eqref{eq:Is}, 
\begin{equation}
  \begin{split}
\frac
{d^s}{dt^s} \mI^{(\beta)}_{m}\!(t){ =}&\!\!\!\!
\sum_{k=\lceil\frac s2\rceil}^{\infty}\frac{t^{2k{-}s}}{(2k{-}s)!}\frac{(\frac{m-1}2{+}k)!}{2\beta^{\frac{m+1}2{+}k}}
   {+}\!\!\!\!\!\!\sum_{k=\lceil\frac{s-1}2\rceil}^{\infty}\!\!\frac{t^{2k{+}1{-}s}}{(2k{+}1{-}s)!}\frac{(m{+}2k{+}1)!}{(\frac{m{-}1}2{+}k{+}1)!}\frac{\beta^{-(\frac{m}2{+}k+1)}}{2^{m+2k+1}}.
    \label{eq:Ims}
  \end{split}
\end{equation}
Evaluated at $t{=}0$, the first term in the first series in
\eqref{eq:Ims} is the only one that survives for even $s$, and yields
$(\frac{m{+}s{-}1}2)!/(2\beta^\frac{m+s+1}2)$, while for odd
$s$, only the first term in the second series makes it and is
$\sqrt{\pi}{(m+s)!}/\{(\frac{m{+}s}2)!\beta^{(\frac{m{+}s{+}1}2)}2^{(m{+}s{+}1)}\}.$
In both cases, 
the surviving term can be re-expressed
as $\frac{\Gamma(\frac{m{+}s{+}1}2)}{2\beta^{\frac{m{+}s{+}1}2}}$, and
so, since $m = 2\ell_{k_1,k_2}^{(j)}{+}1$, we get 
\begin{equation*}
	\begin{split}
  \upmu_{s}{=}&\frac{\alpha}{2\beta}\sum_{j=0}^{\infty}\epsilon_j\cos 2j
              \delta\left(\frac{\eta\psi^{2}}{8}\right)^{j}\sum_{k_1{=}0}^{\infty}\sum_{k_2=0}^{\infty}C(k_1,k_2,j;\psi,\eta)
              {\frac{\Gamma\left(\frac{s}{2}{+}\ell_{k_1,k_2}^{(j)}{+}1\right)}{\beta^{\frac{s}{2}{+}\ell_{k_1,k_2}^{(j)}}}}\\
  {=}&
  \frac{\alpha}{2\beta^{\frac s2+1}}\sum_{j=0}^{\infty}\epsilon_j\cos 2j
\delta\left(\frac\eta{2\beta}\frac{\psi^{2}}{4\beta}\right)^{j}\!\!\sum_{k_1{=}0}^{\infty}\sum_{k_2=0}^{\infty}\left(\frac{\psi^2}{4\beta}\right)^{k_1}\!\!\!\left(\frac\eta{2\beta}\right)^{2k_2}\!\!\!
\frac{\Gamma\left(\frac{s}{2}{+}k_1{+}2k_2{+}2j{+}1\right)}{k_1!(k_1+2j)!k_2!(k_2+j)!} \\
  {=}&\frac{\alpha}{2\beta^{\frac s2+1}}\sum_{j=0}^{\infty}\epsilon_j\cos 2j\delta\left(\frac\eta{2\beta}\frac{\psi^2}{4\beta}\right)^{j}   \sum_{k{=}0}^{\infty}
    \left(\frac\eta{2\beta}\right)^{2k}
  \frac{\Gamma\left(\frac{s}{2}{+}2k{+}2j{+}1\right)}{k!(k+j)!} {  {{}_{1}\F_{1}\left(\frac{s}{2}{{+}}2k{+}2j{{+}}1,2j{+}1,\frac{\psi^2}{4\beta}\right)},}\\
\end{split}
      \end{equation*}
and the result follows. 
\end{proof}

\begin{remark} 
	\label{remark.mus}
	Closed-form expressions are possible for the moments of even order: we write down the first two even-ordered moments that are useful in the context of MOM estimation of parameters. 
\ben
\item The second order raw moment is obtained by noting that  the second raw moment of any $\mN(\mu,\sigma^2)$ random variable is $\mu^2{+}\sigma^2,$ and so
	\begin{equation}
		\upmu_2 {=} \E(X_1^2 {+} X_2^2) {=} \mu_1^2{+}\mu_2^2 {+} \sigma_1^2 {+} \sigma_2^2. 
		\label{eq:simple.mu2}
	\end{equation}
\item The fourth order raw moment of $R$ is given by
	\begin{equation}
		\begin{split}
			\upmu_4 
			{=} \mu_1^4{+}\mu_2^4 {+} 6(\sigma_1^2\mu_1^2{+}\sigma_2^2\mu_2^2){+}3(\sigma_1^4{+}\sigma_2^4) {{+}2\left\{\mu_1^2\mu_2^2{+}(1{+}2\rho^2)\sigma^2_1\sigma_2^2 {+} \mu_1^2\sigma_2^2 {+} \mu_2^2\sigma_1^2  {+} 4\rho\mu_1\mu_2\sigma_1\sigma_2\right\}}.
	\end{split}
		\label{eq:simple.mu4}
	\end{equation}
	To see this, note that $\upmu_4 {=} \E\left\{(X_1^2{+}X_2^2)^2\right\} = \E(X_1^4) {+} \E(X_2^4) {+} 2\E(X_1^2X_2^2)$. Also, the third and fourth raw moments of any $\mN(\mu,\sigma^2)$ random variable $X$ is given by $\E(X^3) {=} \mu^3{+}3\mu\sigma^2$ and $\E(X^4){=} \mu^4{+}6\sigma^2\mu^2+3\sigma^4$. Further, for any bivariate Gaussian random vector $\bX{\sim}\mN_2(\bmu,\bSigma)$ with $\bmu$ and $\bSigma$ as in the statement of Result~\ref{theo:2d.PDF}, the conditional distribution of $X_1$ given $X_2$ is normal with conditional mean $\E(X_1{\mid}X_2){=}\mu_1 {+} \rho\frac{\sigma_1}{\sigma_2}(X_2{-}\mu_2)$ and conditional variance $\sigma_1^2(1 - \rho^2)$. Therefore, 
	\begin{equation*}
		\begin{split}
			\E(X_1^2X_2^2) {=} \E_{X_2}\left[X_2^2\left\{\E(X_1^2{\mid}X_2)\right\}\right]
				       {=}\E_{X_2}\left[X_2^2\left\{\mu_1^2 {+} \rho^2\frac{\sigma_1^2}{\sigma_2^2}(X_2{-}\mu_2)^2 {+}2\rho\frac{\sigma_1}{\sigma_2}\mu_1(X_2-\mu_2)  {+} \sigma_1^2(1{-}\rho^2) \right\}\right]
		\end{split}
		\end{equation*}
	\een
\end{remark}

\subsubsection{Identifiability of parameters}
\label{sec:identifiability}
The second order GBD  ostensibly has five parameters, but lacks identifiability in all of them because 
 $R=\bX^\top\bX{\equiv} \sqrt{\bY^\top\bY}$ for $\bY{=}\bGamma\bX$ where $\bGamma$ is any orthogonal matrix. Consequently, with different choices of $\bGamma$, the second order GBD  can be reparametrized in many ways. For example, a second order GBD random variable with parameters $(\mu_1,\mu_2,\sigma_1,\sigma_2,\rho)$ is functionally and distributionally equivalent to a second order GBD random variable with parameters 
$(\mu^\bullet_1,\mu^\bullet_2,\sigma^\bullet,\sigma^\bullet,\rho^\bullet)$ satisfying
\begin{equation*}
	\begin{split}
	\mu_1^\bullet & = \frac{\mu_1	\sqrt{\sqrt{(\sigma_1^2-\sigma_2^2)^2 +\rho^2\sigma_1^2\sigma_2^2} + \rho\sigma_1\sigma_2}-			\mu_2
	\sqrt{\sqrt{(\sigma_1^2-\sigma_2^2)^2 +\rho^2\sigma_1^2\sigma_2^2} - \rho\sigma_1\sigma_2}}{\sqrt2\sqrt[4]{(\sigma_1^2-\sigma_2^2)^2 +\rho^2\sigma_1^2\sigma_2^2}},\\
	\mu_2^\bullet & = \frac{\mu_1	\sqrt{\sqrt{(\sigma_1^2-\sigma_2^2)^2 +\rho^2\sigma_1^2\sigma_2^2} - \rho\sigma_1\sigma_2}+			\mu_2	\sqrt{\sqrt{(\sigma_1^2-\sigma_2^2)^2 +\rho^2\sigma_1^2\sigma_2^2} + \rho\sigma_1\sigma_2}}{\sqrt2\sqrt[4]{(\sigma_1^2-\sigma_2^2)^2 +\rho^2\sigma_1^2\sigma_2^2}},\\
	\sigma^\bullet & = \frac12\left\{ \sigma_1^2+\sigma_2^2 - \frac{(\sigma_1^2-\sigma_2^2)\rho\sigma_1\sigma_2}{\sqrt{(\sigma_1^2-\sigma_2^2)^2 + \rho^2\sigma_1^2\sigma_2^2}} \right\}, \qquad\mbox{and}\\
		\rho^\bullet &= \frac{(\sigma_1^2-\sigma_2^2)^2 + 2\rho^2\sigma_1^2\sigma_2^2}{(\sigma_1^2+\sigma_2^2)\sqrt{(\sigma_1^2+\sigma_2^2)^2 +\rho^2\sigma_1^2\sigma_2^2} - \rho\sigma_1\sigma_2}.
	\end{split}
\end{equation*}
That is, the underlying complex Gaussian variable has identical variances in both the real and imaginary components. We call this model the Identical Quadrature Components (IQC) model and discuss it further in Section~\ref{sec:IQC}. 

Another equivalent model is the (second order) Beckmann distribution, or the second order GBD with $\rho=0$, with the underlying Gaussian variances given by the eigenvalues of $\bSigma$ and means given by a projection of the original means $(\mu_1,\mu_2)$ and the projection matrix specified by the normalized eigenvectors of $\bSigma$. Specifically, writing $\bSigma\equiv\bGamma\bLambda\bGamma^\top$ in the GBD in terms of its spectral decomposition, we have $\bLambda$ as the diagonal matrix of the eigenvalues of $\bSigma$
\begin{equation}
\frac{\sigma_1^2+\sigma_2^2}2\pm\frac12\sqrt{(\sigma_1^2-\sigma_2^2)^2+4\rho^2\sigma_1^2\sigma_2^2}
\label{eq:eigenvalues.Beckmann}
\end{equation}
and $\bGamma$ as the orthogonal matrix of the corresponding normalized eigenvectors  that are proportional to 
\begin{equation}
\left(\frac{\sigma_1^2-\sigma_2^2}2 \pm \frac12\sqrt{(\sigma_1^2-\sigma_2^2)^2+4\rho^2\sigma_1^2\sigma_2^2}), \rho\sigma_1\sigma_2\right)^\top,
\label{eq:eigenvectors.Beckmann}
\end{equation}
with normalizing constant  
$\frac{(\sigma_1^2-\sigma_2^2)^2}2+2\rho^2\sigma_1^2\sigma_2^2\pm\frac{(\sigma_1^2-\sigma_2^2)}2\sqrt{(\sigma_1^2-\sigma_2^2)^2+4\rho^2\sigma_1^2\sigma_2^2}$. Then the second order GBD model is not indentifiable from the Beckmann distribution with parameters $\bGamma\bmu$ and the normal variance parameters $\frac{\sigma_1^2+\sigma_2^2}2\pm\frac12\sqrt{(\sigma_1^2-\sigma_2^2)^2+4\rho^2\sigma_1^2\sigma_2^2}$. The discussion here also points to the fact that the IQC distribution can be reparametrized in terms of the Beckmann distribution, and the converse also holds.

\subsubsection{Limiting distributions}
We state and prove the following 
\begin{theorem}
	As $\frac{\psi^2}{\beta}\rightarrow\infty$, the density of $R$ is approximately $\mN(\zeta,\tau^2)$ with parameters 
	\begin{equation}
		\zeta = \sqrt[4]{(\mu_1^2+\mu_2^2)^2 + 2(\mu_1^2\sigma_2^2 + \mu_2^2\sigma_1^2 - 2\rho\mu_1\mu_2\sigma_1\sigma_2) + 2\sigma_1^2\sigma_2^2(1-\rho^2)},
		\label{eq:Gauss.mu}
	\end{equation}
	and 
	\begin{equation}
	\tau =\sqrt{\mu_1^2 + \mu_2^2 + \sigma_1^2 + \sigma_2^2 - \zeta^2}.
	\label{eq:Gauss.tau}
\end{equation}
	\label{conjecture}
\end{theorem}
\begin{proof}
	We first show Gaussianity of the second order GBD as $\frac{\psi^2}\beta\rightarrow\infty$. We show that in that case, the $s$th cumulant $\kappa_s\rightarrow0$ for $s\geq3$. To do so, we use  (13.1.4) of \citep{abramowitzandstegun64} which states that $\Gamma(a){}_1\F_1(a,b,z)=\Gamma(b)\exp{(z)}z^{a-b}\{1+\mO(\abs{z}^{-1})\}$. Then, from \eqref{eq:mus},
		  \begin{equation}
    	    \upmu_s{\rightarrow}\left(\frac{\psi^2}{4\beta}\right)^s\frac{\alpha}{2\beta}\sum_{j=0}^{\infty} \sum_{k{=}0}^{\infty}
 \epsilon_j\cos 2j\delta\left(\frac{\psi^2}{4\beta}\frac\eta{2\beta}\right)^{2k+j}
    \frac{\Gamma\left(2j{+}1\right)}{k!(k+j)!} \equiv \left(\frac{\psi^2}{4\beta}\right)^s,
      \label{eq:mus.asymp}
  \end{equation}
  for all raw moments,  with the last reduction in \eqref{eq:mus.asymp} because $\upmu_0\equiv1$, always. 
For $s>1$, the $s$th cumulant in terms of the raw moments is
\begin{equation}
\kappa_{s}=\sum_{k=1}^{s}({-}1)^{k{-}1}(k{-}1)!\B_{s,k}(\upmu_1,\upmu _{2},\ldots ,\upmu _{{s-k+1}}),
\label{eq:kr}
\end{equation}
where $\B_{s,k}(\cdot)$ are the incomplete or partial exponential Bell polynomials\citep{bell34} given by 
\begin{equation}\B_{s,k}(x_{1},x_{2},\dots ,x_{{s-k+1}})=
s!\sum \prod _{i=1}^{{s-k+1}}{\frac {x_{i}^{j_{i}}}{(i!)^{j_{i}}j_{i}!}},
\label{eq:bell}
\end{equation}
where the sum is taken over all sequences $j_1, j_2, \ldots, j_{{s-k+1}}$ of non-negative integers satisfying ${ j_{1}{+}j_{2}{+}\cdots{ +}j_{{s-k+1}}=k,}$ and ${\sum_{m=1}^{{s-k+1}}mj_{m}=s.}$ 
  Incorporating \eqref{eq:mus.asymp} into \eqref{eq:kr} for $s{\geq}3$, and using the Bell polynomial specification of \eqref{eq:bell} yields
\begin{equation}
	\begin{split}
    \kappa_s \rightarrow  \sum_{k=1}^{s}({-}1)^{k{-}1}(k{-}1)!\B_{s,k}\left(\frac{\psi^2}{4\beta},\frac{\psi^4}{4\beta^2},\ldots ,\frac{\psi^{2{s-k+1}}}{4\beta^{s-k+1}} \right) & = \left(\frac{\psi^2}{4\beta}\right)^s 
\sum_{k=1}^{s}({-}1)^{k{-}1}(k{-}1)!\
				     s!\sum \prod _{i=1}^{{s-k+1}}{\frac1{(i!)^{j_{i}}j_{i}!}}\\
				     & =\left(\frac{\psi^2}{4\beta}\right)^s \sum_{k=1}^{s}({-}1)^{k{-}1}(k{-}1)!\B_{s,k}\left(1,1,\ldots,1 \right) \\
	\end{split}
	\label{eq:kr.asymp}
\end{equation}
with the inner sum in the penultimate expression as in \eqref{eq:bell}. Further, from Lemma~\ref{lemma:bellsum}, $\sum_{k=1}^{s}({-}1)^{k{-}1}(k{-}1)!\B_{s,k}\left(1,1,\ldots,1 \right) = 0$
 and so $\kappa_s\rightarrow0$ as $\frac{\psi^2}{\beta}\rightarrow\infty$ for $s\geq3$. From the discussion on Page 49 of \citet{johnsonandkotz70} that follows (14) there, the normal distribution is characterized by zero cumulants of order higher than three, and so the limiting distribution of the second order GBD is indeed normal.

It remains to calculate the parameters of this limiting Gaussian distribution. 
The above arguments also show that $\kappa_2\rightarrow0$ and so the approximation is not precise enough to yield a non-zero asymptotic variance for our limiting distribution. 
At the same time, \eqref{eq:mus}, for odd ordered moments where closed-form expressions are not possible, can suffer from numerical stability, in almost precisely the cases favoring the conditions for the limiting distribution. So, we adopt an indirect approach to calculate the parameters of the limiting Gaussian distribution.
From Remark~\ref{remark.mus}, even-ordered moments of the second order GBD are exactly expressed in closed form. 
Since these are the exact second and fourth ordered moments, we can use them to obtain the parameters of the limiting Gaussian distribution. Proposition~\ref{prop:gaussian.parameters}
in Appendix~\ref{sec:normal} derives the mean and variance of a Gaussian random variable given its second and fourth raw moments. 
The parameters for our limiting Gaussian distribution are then immediate upon inserting \eqref{eq:simple.mu2} and \eqref{eq:simple.mu4} for $\upmu_2$ and $\upmu_4$ in Proposition~\ref{prop:gaussian.parameters} and further simplification.
\end{proof}

Another asymptotic distributional result is obtained as $\bSigma$ tends towards degeneracy. Specifically, we have
\begin{theorem}
	\label{theo:rho.infty}
	As $\rho\rightarrow\pm1$, the limiting density of $R$ is
	\begin{equation}
	\label{eq:rho.infty}
	f_R(r;\lambda,\varsigma,\varepsilon) {=}\sqrt{\frac 2\pi}\frac r\varsigma\left(\frac{r^2{-}\varepsilon}{\varsigma}\right)^{-\frac12}\exp{\left(-\frac {r^2{-}\varepsilon}{2\varsigma}-\frac\lambda2\right)}\cosh{\left(\sqrt{\frac{\lambda (r^2{-}\varepsilon)}{\varsigma}}\right)}\1[r>\sqrt\varepsilon],
	\end{equation}
	where $\varsigma={\sigma_1^2{{+}}\sigma_2^2}$, $	\lambda=\varsigma^{-2}\{\mu_1^2\sigma_1^2{+}\mu_2^2\sigma_2^2{+} 2\mu_1\mu_2\sigma_1\sigma_2\mbox{sign}(\rho)\}$, and $\varepsilon {=} \varsigma^{-1}\{\mu_2^2\sigma_1^2{+}\mu_1^2\sigma_2^2\allowbreak{-}2\mu_1\mu_2\sigma_1\sigma_2\sign(\rho)\}$.
\end{theorem}
\begin{proof}
	From Section~\ref{sec:identifiability}, the value of $R$ and hence its distribution is invariant under orthogonal transformation of the generating bivariate normal random vector $\bX$. Now, for $\bX\sim\mN_2(\bmu,\bSigma)$ with $\bSigma$ as in Result~\ref{theo:2d.PDF}, as $\rho\rightarrow\pm1$, the larger eigenvalue of $\bSigma$ tends to $\varsigma{=}\sigma_1^2{{+}}\sigma_2^2$, while the smaller one tends to 0. The corresponding limiting (orthonormal) eigenvectors are $\varsigma^{-\frac12}\{\sigma_1,\sigma_2\mbox{sign}(\rho)\}^\top$ and $\varsigma^{-\frac12}\{-\sigma_2,\sigma_1\mbox{sign}(\rho)\}^\top$. Consequently, as $\rho\rightarrow\pm1$, the limiting distribution of $R$ matches the distribution of $\sqrt{\varsigma Y_1^2{+}Y_2^2}$ where $
	Y_1{\sim}\mN(\varsigma^{-1}\{\mu_1\sigma_1{+}\mu_2\sigma_2\mbox{sign}(\rho)\}, 1)$, and is independent of $Y_2$ that has point mass at $c {=} \varsigma^{-\frac12}\{\mu_2\sigma_1\mbox{sign}(\rho){-}\mu_1\sigma_2\}.$
Therefore the limiting distribution of $R$ as $\rho\rightarrow\pm1$ is the distribution of the random variable $U=\sqrt{\varsigma W{+}\varepsilon}$ where $U$ is a non-central $\chi^2_{1,\lambda}$ random variable with non-centrality parameter $\lambda$, and $\varepsilon{=}c^2.$ 
The result follows from Theorem~\ref{theo:pdf.ncpchi} in Appendix~\ref{sec:chi} for $k{=}1$ and $\varepsilon,\varsigma,\lambda$ as defined above, and from (10.2.4) of \citet{abramowitzandstegun64} that alternatively represents $\I_{-\frac12}(x) = \frac{\cosh{(x)}}{\sqrt{\frac{x\pi}2}}$.
\end{proof}
The CDF for $R$ is $F_R(r){=}1-\Q_{\frac12}(\sqrt\lambda,\sqrt{(r^2{-}\varepsilon)/\varsigma}),$ folllowing \eqref{eq:cdf.ncpchi} in Appendix~\ref{sec:chi}. Similarly, the moments of $R$ as $\rho\rightarrow\pm1$ can be obtained from Corollary~\ref{cor:moments.cad} there by setting $k{=1}$ and with  $\varepsilon,\varsigma,\lambda$ as in Theorem~\ref{theo:rho.infty}.

\subsubsection{Numerical aspects}
\label{sec:mus.asymptotics}
Computing the density \eqref{eq:2d.PDF}, the CDF \eqref{eq:CDF} or the moments \eqref{eq:mus} is problematic when $\frac{\psi^2}{4\beta}$ or $\abs{\rho}$ is large because of instability in the special functions in these expressions, so we appeal to Theorem~\ref{conjecture} or Theorem~\ref{theo:rho.infty}, as appropriate, in such situations. We study computational asymptotics in Section~\ref{sec:IQC}, and use the asymptotic PDF, CDF and moments when the formulae derived in this section show instability. The exact scenarios with such instability are investigated via simulation in Section~\ref{sec:assessments}.

\section{Application to Special Distributions}
\label{sec:special}
We now illustrate our derivations of Section~\ref{sec:main} on some
envelope distributions that are special cases of the generalized
Beckmann distribution. In some cases, the properties are already known and
our derivations provide the same or alternative characterization,
while in others, our derivations provide additional insights into these  distributions.
\subsection{The Rayleigh distribution}
\label{sec:Rayleigh}
In this case,  $\mu_1{=}\mu_2{=}0,$ $\rho{=}0$ and
$\sigma_1{=}\sigma_2{=}\sigma$. Then, $\tilde{\phi}
{=} 0,$ $\phi {=} \pi/2$, $\alpha{=}1/\sigma^{2}$,
$\beta{=}1/(2\sigma^2),$ $\eta{=}\psi{=}0$, $\I_0(0){=}1$ and $\I_s(0){=}0\,\forall\,
s{\geq}1.$ From Result~\ref{theo:2d.PDF}, 
$f_R(r;\mu,\sigma)=\frac{r}{\sigma^2}\exp{\left(-\frac{r^2}{2\sigma^2}\right)}\1{(r{>}0)}$, which is
the known directly calculated PDF of the Rayleigh distribution. 
Further,
$C(k_1,k_2,j;0,0) {=} 0$, unless $k_1{=}k_2{=}j{=}0$, in which case,
$C(0,0,0;0,0){=}1.$ Then,
$F_R(u;\bmu,\allowbreak\bSigma){=}1-\exp{\{-u^2/(2\sigma^2)\}}$, for $u{\geq} 0$,
which we know is the directly calculated CDF of the Rayleigh distribution with scale
parameter $\sigma$. The only non-zero terms in the MGF
of~\eqref{eq:mgf} are when 
$k_1{=}k_2{=}j{=}0.$ Then, from~\eqref{eq:mus},  $\mI_1^{(\beta)}
(t)={_1\F_1}(1,\frac12,\frac{\sigma^2t^2}2) + \sigma
t\sqrt\frac{\pi}2 {}_1\F_1(\frac32,\frac32,\frac{\sigma^2t^2}2)$, and
$\M_R(t) {=}  \alpha \mI_1^{(\beta)}(t)$. From 
Corollary~\ref{lemma:1F1} in Appendix~\ref{sec:proof.lemma1F1}, we have that  
${{}_1\F_1(1,\frac12,\frac{\sigma^2t^2}2){=} 1 {+}\sigma^2t^2
{}_1\F_1(1,\frac32,\frac{\sigma^2t^2}2),}$ and
 from (13.1.27) and (13.6.19)
of~\citet{abramowitzandstegun64},
${_1\F_1(1,\frac32,\frac{\sigma^2t^2}2) {=}
\exp{\left(\frac{\sigma^2t^2}2\right)}
{}_1\F_1(\frac12,\frac32,-\frac{\sigma^2t^2}2)}$ and   
${{}_1\F_1(\frac12,\frac32,-\frac{\sigma^2t^2}2) {=}
\sqrt{\frac\pi2}\frac1{\sigma t} \mathrm{erf}\left(\frac{\sigma t}{\sqrt{2}}\right)},$ while from (13.6.12) of~\citet{abramowitzandstegun64}, ${{}_1\F_1(\frac32,\frac32,\frac{\sigma^2t^2}2) =
\exp{\left(\frac{\sigma^2t^2}2\right)}}$. Combining, 
$$\M_R(t){=} 1+\sigma
t\sqrt{\frac{\pi}{2}}\exp{\left(\frac{\sigma^2t^2}{2}\right)}\left\{\mathrm{erf}\left(\frac{\sigma
      t}{\sqrt{2}}\right)+1\right\},$$ which is the known directly
calculated MGF of the Rayleigh distribution. Concluding, the $s$th
Rayleigh  raw moment is $\upmu_s= (\alpha/2\beta) C(0,0,0,0,0)
\Gamma(s/2+1)/\beta^{\frac s2} = 2^{s/2}\sigma^s\Gamma(s/2+1),$  since
$_1\F_1(a,b,0)=1$, which matches results obtained independently of our
derivations from the formulae in Section~\ref{sec:main}.


\subsection{The Rice distribution}
\label{sec:Rice}
In this case, $\rho{=}0,$ $\sigma_1{=}\sigma_2{=}\sigma$, and
$\mu_1=\nu\cos\xi$, $\mu_2=\nu\sin\xi$ in polar form. Then
$\alpha=\frac{1}{\sigma^2}\exp{\left(-\frac{\nu^2}{2\sigma^2}\right)},$
$\beta=\frac{1}{2\sigma^2}\,\,,
\psi={\nu}/{\sigma^2}$, $\eta{=}0$ (and so the contribution of $\cos 2j\delta$ in any of the quantities is immaterial unless $j=0$). So
$$f\!_R(r;\bmu, \bSigma) {=}\frac r{\sigma^2}\exp{\left(\!-\frac{\nu^2{+}r^2}{2\sigma^2}\!\right)}\I_{0}\left(\frac{\nu  r}{\sigma^2} \right)\1[r{>}0],$$
the known directly calculated Rice($\sigma,\nu$) PDF. Also, as in Section~\ref{sec:Rayleigh}, 
$C(k_1,k_2,j;\psi,\eta)$ makes 
a positive contribution to the series in~\eqref{eq:CDF},~\eqref{eq:MGF}
or~\eqref{eq:mus} only  when $k_2{=}j{=}0$. Now, $C(k_1,0,0;\psi,0) {=}
\frac{\psi^{2k_1}}{4^{k_1}(k_1!)^2}$, so that $T_1(k_1,0,0;\psi,0) =
\frac{\psi^{2k_1}}{4^{k_1}k_1!\beta^{k_1}}$ and
  $T_2(u,k_1,0,0,k;\psi,0) =
  \frac{\psi^{2k_1}}{4^{k_1}k_1!(k_1-k+1)!\beta^{k-1}}
  u^{2(k_1+1-k)}.$ 
  From \eqref{eq:CDF},
  \begin{equation}
    \begin{split}
  F_R(u;\sigma,\nu) &{=} \exp{\left(\!-\frac{\nu^2}{2\sigma^2}\!\right)}
  \sum_{k_1=0}^\infty\frac{\psi^{2k_1}}{4^{k_1}k_1!\beta^{k_1}} {-}
  \exp{\left(-\frac{u^2}{2\sigma^2}\right)}\sum_{k_1=0}^\infty\sum_{k=1}^{k_1+1} \frac{\psi^{2k_1}u^{2(k_1{+}1{-}k)}}{4^{k_1}k_1!(k_1{-}k{+}1)!\beta^{k-1}}
      .
      \end{split}
  \label{eq:Rice.CDF1}
\end{equation}
Also, from the alternative representation of the CDF
in~\eqref{eq:alt.CDF}, we get
\begin{equation}
\begin{split}
  F_R(u;\sigma,\nu) 
                    {=}
                    \frac{\alpha}{2\beta}
                    \sum_{k{=}0}^{\infty}\!\!\frac{C(k,0,0,\psi,0)}{\beta^{k}}{\gamma}(k{+}1,\beta{u^2})
                                       {=}\exp{\left(\!{-}\frac{\nu^2}{2\sigma^2}\!\right)}
                         \sum_{k{=}0}^{\infty}\frac{1}{k!} \frac{{\gamma}(k+1,\frac{u^2}{2\sigma^2})}{\Gamma(k+1)}\left(\frac{\nu^{2k}}{2^k\sigma^{2k}}\right){=}1{-}\Q_1\left(\frac\nu\sigma,\frac{u}{\sigma}\right),
\end{split}
\label{eq:Rice.CDF2}
                         \end{equation}
where $\Q_1(\cdot,\cdot)$ is the generalized Marcum $\Q$-function of the
first order~\citep{marcum60}, and the reduction to it follows from
(2.12) of~\citet{andrasetal11}. The CDF~\eqref{eq:Rice.CDF2} is
preferred over the one in~\eqref{eq:Rice.CDF1} because high-precision
algorithms for $\Q_1(\cdot,\cdot)$ exist in many standard software.

For the MGF, we have from
  \eqref{eq:mgf}, for all $t{\in}\R$,
  \begin{align*}
    \M_R(t)= \exp{\left(-\frac{\nu^2}{2\sigma^2}\right)} \Bigg\{
    \sum_{k=0}^{\infty} \frac{(\frac{\nu}{\sigma})^{2k}}{2^kk!}{}_{1}\F_{1}\!\left(k{+}1,\frac{1}{2},\frac{\sigma^2t^2}{2}\right){ +
  }  \sqrt{\frac{\pi}{2}}\frac{\sigma
                                                                                                                                     t}2\sum_{k=0}^{\infty}\frac{(\frac{\nu}{\sigma})^{2k}}{8^k(k!)^2\,}\frac{(2k{+}2)!}{(k{+}1)!}{{}_{1}\F_{1}\left(\frac{2k{+}3}{2},\frac{3}{2},\frac{\sigma^2t^2}{2}\right)\Bigg\}},
\end{align*}
and, from \eqref{eq:mus},
 \begin{align*}
\upmu_s & 
=
\exp{\left(-\frac{\nu^2}{2\sigma^2}\right)}\sum_{k_1=0}^{\infty}\frac{\nu^{2k_1}}{4^{k_1}\sigma^{4k_1}{k_1!^2\,}}\frac
{\Gamma(k_1+1+\frac s2)}{\beta^{\frac s2{+}k_1}}=\beta^{{-}\frac s2}
\exp{\left(-\frac{\nu^2}{2\sigma^2}\right)}\Gamma\left(\frac
  s2{+}1\right) {}_1\F_1\left(\frac
                                                                                                                                                                        s2{+}1,1,\frac{\nu^2}{2\sigma^2}\right).
 \end{align*}

From (13.1.27) and (13.6.9) of~\citet{abramowitzandstegun64},
\begin{align*}
\exp{({-}\frac{\nu^2}{2\sigma^2})}{}_{1}\F_{1}\left(\frac{s}{2}{+}1,1,\frac{\nu^2}{2\sigma^2}\right)&{=} {}_{1}\F_{1}\left(-\frac{s}{2},1,{-}\frac{\nu^2}{2\sigma^2}\right){=}\Lag_{\frac s2}\left({-}\frac{\nu^2}{2\sigma^2}\right),
  \end{align*}
where $\Lag_q(\cdot)$ is the $q$th (simple) Laguerre polynomial. Therefore,
\begin{equation}
  \upmu_s= 2^{\frac{s}{2}}\sigma^{s}\Gamma\left(\frac{s}2+1\right)\Lag_{\frac s2}\left(\!{-}\frac{\nu^2}{2\sigma^2}\right).
  \label{eq:Rice.mus}
\end{equation}

\subsubsection{Rice distribution asymptotics}
\label{sec:rice.asymptotics}
The Rice distribution is stated to be normally distributed as $\sigma\rightarrow0$, but we are unaware of a formal proof in support of this statement, with many proofs only proving that the right tails of the Rice density match those of the limiting distribution. Theorem~\ref{conjecture} provides a complete proof and shows that as $\frac\nu{\sigma^2}\rightarrow\infty$, then $R\stackrel{a}\sim\mN(\nu,2\sigma^2)$. Given the historical and practical importance of the Rice distribution, we provide another statement and proof here.
\begin{theorem}
	\label{theo:rice.asymptotics}
	Let $R$ be a random variable from the Rice density with parameters $\nu$ and scale parameter $\sigma$. As $\frac\nu\sigma\rightarrow\infty$, $\nu^s\rightarrow\infty$ at a slower rate than $\frac\nu\sigma$ for all positive integers $s$, we have that $R\stackrel{a}\sim \mN(\nu,2\sigma^2)$.
\end{theorem}
\begin{proof}
	We use our derived raw moments in our proof. From our derivation of~\eqref{eq:Rice.mus}, 
	\begin{equation*}
	\upmu_s= 2^{\frac s2}\sigma^s\Gamma\left(\frac s2+1\right) {}_1\F_1\left(-\frac s2, 1, -\frac{\nu^2}{2\sigma^2}\right).
	\end{equation*}
	On the other hand, the $s$th raw moment of a  $\mN(\mu,\sigma^2)$ random variable can be  expressed~\citep{winkelbauer14} as
	\begin{equation}
		\upmu^{(\Phi_{\mu,\sigma^2})}_s = \begin{cases}
		\mu\sigma^{s-1}2^{\frac{s+1}2} \frac{\Gamma(\frac s2+1)}{\sqrt\pi} {}_1\F_1\left(\frac{1- s}2,\frac32,-\frac{\mu^2}{2\sigma^2}\right), & \mbox{ for  odd $s\geq 0$} \\
			\sigma^s 2^{\frac s2}\frac{\Gamma(\frac{s+1}2)}{\sqrt\pi}{}_1\F_1\left(-\frac s2,\frac12,-\frac{\mu^2}{2\sigma^2}\right), & \mbox{ for even $s\geq 0$}.\\
		\end{cases}
	\end{equation}
For odd $s$, we have
\begin{equation*}
	\begin{split}
	 \left|\upmu_s - \upmu_s^{(\Phi_{\nu,2\sigma^2})}\right|
	 &\hspace{0.5in}=2^\frac s2\sigma^s \left|\Gamma\left(\frac s2+1\right){}_1\F_1\left(-\frac s2, 1,-\frac{\nu^2}{2\sigma^2}\right) -\frac{2^{\frac s2}\nu\Gamma(\frac s2+1)}{\sigma\sqrt\pi}  {}_1\F_1\left(\frac{1- s}2,\frac32,-\frac{\nu^2}{4\sigma^2}\right) \right| \\
	 &\hspace{0.5in} =  2^\frac s2\sigma^s \left|\left(\frac{\nu^2}{2\sigma^2}\right)^{\frac s2}
		{-}2^{\frac {s-1}2}\nu\left(\frac{\nu^2}{4\sigma^2}\right)^{\frac{s-1}2}\right|	
		\left\{1	+\mO\left(\frac{\sigma^2}{\nu^2}\right)\right\} = \mO\left(\frac{\sigma^2}{\nu^2}\right),
	\end{split}
\end{equation*}
where the last line follows from (13.1.5) of \citep{abramowitzandstegun64} and from using $\Gamma(\frac32)=\frac{\sqrt{\pi}}2$.
Using similar arguments for even $s$, we get that
\begin{equation*}
	\begin{split}
		\left|\upmu_s {-} \upmu_s^{(\Phi_{\nu,2\sigma^2})}\right| &{=}2^\frac s2\sigma^s \left|\Gamma\left(\frac s2+1\right){}_1\F_1\left(-\frac s2, 1,-\frac{\nu^2}{2\sigma^2}\right) {-}\frac{2^{\frac s2}\Gamma(\frac {s+1}2)}{\sqrt\pi}  {}_1\F_1\left(-\frac{s}2,\frac12,-\frac{\nu^2}{4\sigma^2}\right) \right| \\
	  & {=}2^\frac s2\sigma^s \left|\left(\frac{\nu^2}{2\sigma^2}\right)^{\frac s2}	{-}2^{\frac {s}2}\left(\frac{\nu^2}{4\sigma^2}\right)^{\frac{s}2}\right|	+\mO\left(\frac{\sigma^2}{\nu^2}\right) = \mO\left(\frac{\sigma^2}{\nu^2}\right).			
	\end{split}
\end{equation*}
Therefore, all even and odd moments of the Rice density with parameters $(\sigma,\nu)$ converge to those of $\mN(\nu,2\sigma^2)$ as $\frac\nu\sigma\rightarrow\infty$, proving the theorem.
\end{proof}

We conclude with two remarks. The condition in Theorem~\ref{theo:rice.asymptotics} may be slightly weaker than that in Theorem~\ref{conjecture}. Also, in the context of Ricean fading~\citep{abdietal01,lindsey64} the quantity $K\doteq\frac{\nu^2}{2\sigma^2}$ is often referred to as the {\em shape parameter} of the Rice distribution and along with the raw second moment $\Omega\doteq\nu^2{+}2\sigma^2$, also (perhaps somewhat confusingly to probabilists and statisticians) called the {\em scale parameter}, provides an alternative parametrization of the distribution. There also exists an alternative form of the Rice distribution due to~\citet{rangaswamyetal93}, that we characterize further in Sections~\ref{sec:IQC} and \ref{sec:back}.

\subsection{The Hoyt distribution}
\label{sec:Hoyt}
The Hoyt distribution is when $\mu_1{=}\mu_2{=}0$, $\rho{=}0$, but
$\sigma_2\geq\sigma_1$. So it is a slight generalization of the
Rayleigh distribution. For this distribution, therefore, we have
$\alpha=\frac{1}{\sigma_1\sigma_2}$,
$\beta=\frac{1}{4}\left(\frac{1}{\sigma^2_{1}}{+}\frac{1}{\sigma^2_{2}}\right)$,
$\psi=0,$
$\eta=\frac{\sigma^2_2{-}\sigma^2_1}{4\sigma^2_1\sigma^2_{2}}$ and
$\,\,\delta=\frac{\pi}{2}$. This means that the only terms that show
up in the series in the PDF is when $j{=}0$. Therefore,
using our formula, the PDF of the Hoyt distrbution
is $$f_{R}(r){=}\frac{r}{\sigma_1\sigma_2}\exp{\left\{{-}\frac{1}{4}\left(\frac{1}{\sigma^2_{1}}{+}\frac{1}{\sigma^2_{2}}\right)
    r^2\right\}}\I_{0}\left(\frac{\sigma^2_2{-}\sigma^2_1}{4\sigma^2_1\sigma^2_{2}}
  r^2\right)\!\1[r{>}0],$$
which is similar to the one specified in~\citet{aaloetal07}. Further,
the terms show up in the series expressions for the CDF, the MGF or the
moments in Section~\ref{sec:main} only when $k_1{=}j{=}0$ and $k_2\geq 0$.  In this case,
$C(0,k_2,0;\psi,\eta)=\frac{1}{4^{k_2}k_2!^2}\left(\frac{(\sigma^2_2-\sigma^2_1)^2}{16\sigma^4_1\sigma^4_{2}}\right)^{k_2}$. Further,
$\ell_{k_1,k_2}^{(j)}=2k_2$, so from Proposition~\ref{theo:altCDF}, we
get the CDF to be
\begin{equation*}
  \begin{split}
&F_R(r;\sigma_1,\sigma_2){=\frac{1}{2\sigma_{1}\sigma_{2}}
    \sum_{k_2=0}^{\infty}\frac{1}{4^{k_2}k_2!^2}\left(\frac{(\sigma^2_2-\sigma^2_1)^2}{16\sigma^4_1\sigma^4_{2}}\right)^{k_2}\!\frac{u^{4k_2+2}}{2k_2+1}
        {}_{1}\F_{1}\left(2k{+}1,2k{+}2,-\frac{1}{4}\left(\frac{1}{\sigma^2_{1}}{+}\frac{1}{\sigma^2_{2}}\right)u^2\right)},
    \end{split}
\end{equation*}
while the MGF, from Theorem~\ref{theo:mgf}, is
\begin{align*}
\M_R(t)=&\frac{2\sigma_{1}\sigma_{2}}{\sigma^2_{1}+\sigma^2_{2}}
          \sum_{k=0}^{\infty} \left(\frac{(\sigma_2^2
          -\sigma_1^2)^2}{4(\sigma_1^2{+}\sigma_2^2)^2}\right)^{k}\frac{\Gamma(2k{+}1)}{k!^2}{\times\,{}_{1}\F_{1}\left(2k{+}1,\frac{1}{2},\frac{t^2}{\frac{1}{\sigma^2_1}{+}\frac{1}{\sigma^2_2}}\right)}\\&\quad+\frac{4t\sigma^2_{1}\sigma^2_{2}}{(\sigma_{1}{+}\sigma_{2})^{3/2}}\sum_{k=0}^{\infty}\left(\frac{(\sigma_2^2{-}\sigma_1^2)^2}{4(\sigma_1^2{+}\sigma_2^2)^2}\right)^{k}\frac{\Gamma\left(\frac{2k{+}3}{2}\right)}{k!^2}{\times\,{}_{1}\F_{1}\left(\frac{2k{+}3}{2},\frac{3}{2},\frac{t^2}{\frac{1}{\sigma^2_1}{+}\frac{1}{\sigma^2_2}}\right).}
\end{align*}
Further, let $Q^{(k)}_{\sigma_1,\sigma_2} {=}
\frac1{k!}{\frac{(\sigma^2_2{-}\sigma^2_1)^{2k}}{4^k\sigma^{4k}_1\sigma^{4k}_{2}}}$.
Also, define the {\em generalized rising factorial} by $x_{(a)}^{(n)}{\doteq}
x(x+a)(x+2a)\hdots(x+(n{-}1)a).$ Then, since $\psi{=}0$ and
$_1\F_1(a,b,0){=}1,$ Corollary~\ref{theo:mus} yields
\begin{align*}
  \upmu_s{=}\frac{1}{\sigma_1\sigma_2}\sum_{k=0}^{\infty}{k!^2\,}\frac{Q^{(k)}_{\sigma_1,\sigma_2}}{16^kk!}
              \Gamma\left(2k{{+}}\frac{s}2{{+}}1\right)\frac{2^{4k{+}s{+}1}}{\left(\frac{1}{\sigma^2_{1}}{+}\frac{1}{\sigma^2_{2}}\right)^{(2k{{+}}1{{+}}\frac{s}{2})}}
  =&\frac{1}{\sigma_1\sigma_2}\frac{2^{s{+}1}\Gamma\left(\frac{s}{2}{+}1\right)}{\left(\frac{1}{\sigma^2_1}{+}\frac{1}{\sigma^2_2}\right)^{1{+}\frac{s}{2}}}\sum_{k=0}^{\infty}\frac{\left(\frac
     s2{+}1\right)^{(2k)}_{(1)}}{k!} Q^{(k)}_{\sigma_1,\sigma_2}\\
  =&\frac{(2\sigma_1\sigma_2)^{s{+}1}}{\left(\sigma^2_1{+}\sigma^2_2\right)^{1{+}\frac{s}{2}}}\Gamma\left(\frac{s}{2}{+}1\right)\sum_{k=0}^{\infty}\frac{2^{2k}\left(\frac
     s4
     \right)^{(2k{+}1)}_{(\frac12)}}{k!}\frac{\left(\frac{(\sigma^2_2{-}\sigma^2_1)^2}{4(\sigma^2_1{+}\sigma^2_{2})^2}\right)^{k}}{k!\,}\\
  =&\frac{(2\sigma_1\sigma_2)^{s{+}1}\Gamma\left(\frac{s}{2}{+}1\right)}{\left(\sigma^2_1{+}\sigma^2_2\right)^{1{+}\frac{s}{2}}}\sum_{k=0}^{\infty}\!\!\frac{\left(\frac s4{+}\frac12\right)^{(k)}_{(1)}\left(\frac s4{+}1\right)^{(k)}_{(1)}}{k!}\frac{\left(\frac{\sigma^2_2{-}\sigma^2_1}{\sigma^2_1{+}\sigma^2_{2}}\right)^{2k}}{k!\,}\\=&\frac{(2\sigma_1\sigma_2)^{s{+}1}}{\left(\sigma^2_1{+}\sigma^2_2\right)^{1{+}\frac{s}{2}}}\Gamma\left(\frac{s}{2}{+}1\right){}_{2}\F_{1}\left(\frac{s}{4}{+}\frac{1}{2},\frac{s}{2}{+}1,1,\frac{(\sigma^2_2{-}\sigma^2_{1})^2}{(\sigma^2_1{+}\sigma^2_{2})^2}\right),   
\end{align*}
where $_2\F_1(a,b,c,z)$ is the hypergeometric function~\citep{erdelyietal53,whittakerandwatson1915}, and we thus arrive at a formula that matches the independently-derived formula in~(7) of ~\citet{coluccia13}.
\subsection{The Beckmann distribution}
\label{sec:Beckmann}
The Beckmann distribution is a special case of the generalized form,
 when  $\rho{=}0$. Without loss of generality (WLOG), let $\sigma_2{\geq}\sigma_1$. Then we have
$\alpha{=}\frac{1}{\sigma_1\sigma_2}\exp{\left\{-\frac{\mu^2_1}{2\sigma^2_{1}}{+}\frac{\mu^2_2}{2\sigma^2_{2}}\right\}}$,
$\delta{=}\tilde\phi{-}\frac{\pi}{4}
,$
$\beta{=}\frac{1}{4}\left(\frac{1}{\sigma^2_{1}}{+}\frac{1}{\sigma^2_{2}}\right),\,\,
\psi{=}\frac{\sqrt{\mu^2_1\sigma^4_{2}+\mu^2_2\sigma^4_{1}}}{\sigma^2_1\sigma^2_{2}}\,$
and $\eta{=}\frac{\sigma^2_2-\sigma^2_1}{4\sigma^2_1\sigma^2_{2}}$. 
Here, $\tilde\phi$ is such that $\cos\tilde\phi = \mu_2\sigma_1^2/\sqrt{\mu_1^2\sigma_2^2 + \mu_2^2\sigma_1^2}$ and $\sin\tilde\phi = \mu_1\sigma_2^2/\sqrt{\mu_1^2\sigma_2^2 + \mu_2^2\sigma_1^2}$ 
Then, all terms involving all $k_1,k_2,j$ contribute to the
series. Further, 
\begin{equation}
  \begin{split}
    C(k_1,k_2,j;\psi,\eta){=}\frac{\left(\frac{\mu^2_1\sigma^4_{2}{+}\mu^2_2\sigma^4_{1}}{\sigma^4_1\sigma^4_{2}}\right)^{k_{1}}\left(\frac{(\sigma^2_2{-}\sigma^2_1)^2}{16\sigma^4_1\sigma^4_{2}}\right)^{k_2}}{4^{k_1{+}k_2}k_1!(k_1{+}2j)!k_2!(k_2{+}j)!}.
    \end{split}
  \label{eq:Beckmann.C}
\end{equation}
There are no further simplifications of the formulae for the PDF, CDF,
MGF or moments possible here, so we refer back to
Section~\ref{sec:main} for the specific formulate for each of these
quantities, with the above values inserted in those formulae.

\subsection{The identical quadrature components (IQC) model}
\label{sec:IQC}
This case occurs when the underlying complex Gaussian random variable
has identical but correlated quadratures, that is,
$\sigma_1{=}\sigma_2{\equiv}\sigma$. Then, writing $\mu_1=\nu\cos\xi$ and $\mu_2=\nu\sin\xi$ and following the definitions in Section~\ref{sec:fns}, 
$\alpha = 2\beta\exp{\left\{-\beta\nu(1-\rho\sin2\xi )\right\}}$, with 
$\beta{=}\frac12\sigma^{-2}(1{-}\rho^2)^{-1}$, $\eta{=}\rho\beta$, 
$\psi = 2\beta\nu\sqrt{1+\rho^2-2\rho\sin2\xi}$
and $\delta{=}\tilde\phi{-}\frac{\pi}{4}$, where 
$\tilde\phi$ satisfies $\cos\tilde\phi=\frac{\sin\xi-\rho\cos\xi}{\sqrt{1+\rho^2-2\rho\sin2\xi}}$, and $\sin\tilde\phi=\frac{\cos\xi-\rho\sin\xi}{\sqrt{1+\rho^2-2\rho\sin2\xi}}$. 
Then $\cos2j\delta = (-1)^{\frac{j-1}2}\sin2j\tilde\phi$ when $j$ is odd and $\cos2j\delta = (-1)^{\frac j2}\cos2j\tilde\phi$ when $j$ is even. From de Moivre's formula, the binomial series expansion and equating the real and imaginary parts, we have the following representations for 
$
\cos n\theta  =\sum_{k=0}^{\lfloor \frac n2 \rfloor}(-1)^{k}{n \choose 2k}\cos^{n-2k}\theta \sin^{2k}\theta,
$
\mbox{ and }
$
\sin n\theta  =\sin\theta\sum_{k=0}^{\lfloor \frac {n-1}2 \rfloor}(-1)^{k}{n \choose 2k+1}\cos^{n-2k-1}\theta \sin^{2k}\theta
	$
	where we use $\lfloor x \rfloor$ to denote the largest integer not exceeding $x$. Therefore $\cos2j\delta$ can be written in terms of a series involving the product of powers of $\sin2\tilde\phi$ and $\cos2\tilde\phi$. From the definition of $\tilde\phi$ and some trigonometry identities, $\sin2\tilde\phi = \frac{(1+\rho^2)\sin2\xi-2\rho}{1+\rho^2-2\rho\sin2\xi}$ and $\cos2\tilde\phi =  \frac{-(1-\rho^2)\cos2\xi}{1+\rho^2-2\rho\sin2\xi}$.
	Also,  
$$C(k_1,k_2,j;\psi,\eta){=}\frac{\nu^{2k_1}\rho^{2k_2}\beta^{2k_1+2k_2}(1+\rho^2-2\rho\sin2\xi)^{k_1}}{4^{k_2}k_1!(k_1{+}2j)!k_2!(k_2{+}j)!}.$$ As in the Beckmann case, no additional simplifications in the 
formulae from Section~\ref{sec:main} are generally possible and these values are
inserted to get expressions for the different quantities. 
\subsubsection{Identifiability issues}
The polar representation of $\mu_1,\mu_2$ and the consequent specification of $\alpha,\beta,\delta,\eta,\psi$ allows us to see additional identifiability issues in the parameters.

\subsubsection{Special cases} When additionally $\mu_1{=}\mu_2{\equiv}\mu$, we get $\psi{=}\sqrt2\mu\sigma^{-2}(1{+}\rho)^{-1},$ 
$\delta{=}0$, and
$$C(k_1,k_2,j;\psi,\eta)=\frac{\mu^{2k_1}\rho^{2k_2}}{4^{k_1{+}2k_2}k_1!k_2!(k_1{+}j)!(k_2{+}2j)!\sigma^{4k_1+k_2}(1{-}\rho^2)^{k_1+k_2}}.$$
In this  case, the PDF is
\begin{equation*}
  \begin{split}
  f_R(r;\mu,\sigma,\rho){=}\frac r{\sigma^2\sqrt{1{-}\rho^2}}&
\exp{\left\{-\frac{\mu^2(1{-}\rho)
                          {+}r^2}{2\sigma^2(1{-}\rho)^2}\right\}}
    {
\sum_{j=0}^\infty\I_{2j}\left( \frac{\sqrt{2}\mu  r}{\sigma^2(1{+}\rho)}  \right)
\I_j\left( \frac{\rho  r^2}{2\sigma^2(1{-}\rho^2)}  \right)\1[r{>}0].}
\end{split}
    \label{eq:gen.Rice}
\end{equation*}
When additionally $\mu{=}0$, we get another generalization of
the Rayleigh distribution, because then the only terms involving
$j$ in the series contribute when $j{=}0$. In that case, the PDF is
\begin{equation*}
  \begin{split}
  f_R(r;\sigma,\rho){=}\frac r{\sigma^2\sqrt{1{-}\rho^2}}
\exp{\left\{-\frac{r^2}{2\sigma^2(1{-}\rho)^2}\right\}}
    \I_0\left( \frac{\rho  r^2}{2\sigma^2(1{-}\rho^2)}  \right)\1[r{>}0],
    \end{split}
\label{eq:gen.Rayleigh}
\end{equation*}
or an alternative form of the Rice
PDF derived by~\citet{rangaswamyetal93}. 
We subsume further discussion of this case in the next section.
\subsection{The background signal-free model}
\label{sec:back}
The background case is a slight generalization of the Hoyt density, 
and  occurs when $\mu_1{=}\mu_2{=}0$ in our generalized Beckmann
setup. Then,
$\alpha{=}\frac{1}{\sigma_1\sigma_2\sqrt{1-\rho^2}},\allowbreak\beta{=}\frac{\sigma_1^2{+}\sigma_2^2}{4\sigma_1^2\sigma_2^2(1-\rho^2)},\,\psi{=}0,$
and $\eta{=}\frac{\sqrt{(\sigma^2_2-\sigma^2_1)^2+4\sigma^2_1\sigma^2_2\rho^2}}{4\sigma^2_1\sigma^2_2(1-\rho^2)}.$ Also $\tilde{\phi}{=}\frac{\pi}{2}$ and
$\phi{=}\text{arctan}\left(\frac{2\rho\sigma_1\sigma_2}{\sigma^2_{2}-\sigma^2_{1}}\right)$. Then
$k_1{=}0,j{=}0$ is the only time that the terms in the series
contribute. Therefore,
\begin{align*}
  f_{R}&(r;\sigma_1,\sigma_2,\rho){=}\frac{r}{\sigma_1\sigma_2\sqrt{1{-}\rho^2}},\exp{\left(-\frac{(\sigma_1^2{+}\sigma_2^2)r^2}{4\sigma_1^2\sigma_2^2(1{-}\rho^2)}\right)}\I_0\left(\frac{r^2\sqrt{(\sigma^2_2{-}\sigma^2_1)^2{+}4\sigma^2_1\sigma^2_2\rho^2}} {4\sigma^2_1\sigma^2_2(1{-}\rho^2)}\right)\1[r{>}0].
\end{align*}
Also
$C(0,k,0;0,\eta){=} \frac{1}{64^k(k!)^2}\frac{\left\{(\sigma^2_2{-}\sigma^2_1)^2{+}4\sigma^{2}_1\sigma^{4}_2\rho^2\right\}^k}
  {\sigma^{4k}_1\sigma^{4k}_2(1{-}\rho^2)^{2k}},$ and
from Proposition~\ref{theo:altCDF}, we get the CDF
\begin{equation}
  \begin{split}
    F_R(u;\sigma_1,\sigma_2,\rho)&{=}\frac{u^2}{2\sigma_1\sigma_2\sqrt{1{-}\rho^2}}
      \sum_{k{=}0}^\infty\Bigg\{
                                   \frac{u^{4k}}{64^k(k!)^2(2k{+}1)}{\frac{\left\{(\sigma^2_2{-}\sigma^2_1)^2{+}4\sigma^{2}_1\sigma^{4}_2\rho^2\right\}^k}    {\sigma^{4k}_1\sigma^{4k}_2(1{-}\rho^2)^{2k}}}{
{}_1\F_1\left(2k+1,2k+2, -\frac{u^2(\sigma_1^2{+}\sigma_2^2)}{4\sigma_1^2\sigma_2^2(1-\rho^2)}\right)\Bigg\},}
  \end{split}
  \label{eq:back.CDF}
\end{equation}
and the MGF, from Theorem~\ref{theo:mgf}, is
\begin{align*}
  \M_R(t)=\frac{2\sigma_{1}\sigma_{2}\sqrt{(1-\rho^2)}}{\sigma^2_{1}+\sigma^2_{2}}\sum_{k=0}^{\infty}\frac{\left\{(\sigma^2_2{-}\sigma^2_1)^2{+}4\sigma^{2}_1\sigma^{2}_2\rho^2\right\}^k}
            {4^k(k!)^2(\sigma_1^2{+}\sigma_2^2)^{2k}}
            & \left\{ \Gamma(2k{+}1)\,
            {}_{1}\F_{1}\left(2k{+}1,\frac{1}{2},\frac{t^2\sigma_1^2\sigma_2^2(1{-}\rho^2)}{\sigma^2_1{+}\sigma^2_2}\right)
\right.
  \\&
  \pushright{{+}
  \frac{2t\sigma_{1}\sigma_{2}(1-\rho^2)^{1/2}}{\sqrt{(\sigma^2_{1}+\sigma^2_{2})}}
  \Gamma\left(2k{+}\frac{3}{2}\right)\,}\\ &
\pushright{{}_{1}\F_{1}\left(2k{+}\frac{3}{2},\frac{3}{2},
\frac{t^2\sigma_1^2\sigma_2^2(1{-}\rho^2)}{\sigma^2_1{+}\sigma^2_2}  \right)\Bigg\}.}
\end{align*}
Finally, Corollary~\ref{theo:mus} gives
\begin{equation}
  \begin{split}
   \upmu_s &=\frac{2^{s+1}(\sigma_1\sigma_2)^{1+s}(1{-}\rho^2)^{\frac{s{+}1}2}}{
             (\sigma_1^2{+}\sigma_2^2)^{\frac s2{+}1} }{\sum_{k=0}^{\infty}\frac{1}{4^k(k!)^2}\frac{\left\{(\sigma^2_2{-}\sigma^2_1)^2{+}4\sigma^{2}_1\sigma^{2}_2\rho^2\right\}^k}
             {(\sigma_1^2+\sigma_2^2)^{2k}}
             \Gamma\left(2k{+}\frac{s}2{+}1\right).}\\
             &= \frac{2^{s+1}(\sigma_1\sigma_2)^{1+s}(1{-}\rho^2)^{\frac{s{+}1}2}}{
             (\sigma_1^2{+}\sigma_2^2)^{\frac s2{+}1} }\Gamma\left(\frac{s}2{+}1\right)\,{{}_{2}F_{1}\left(\frac{s}4{+}\frac{1}2,\frac{s}4{+}1,1,\frac{(\sigma^2_2{-}\sigma^2_1)^2{+}4\sigma^{2}_1\sigma^{2}_2\rho^2}           {(\sigma_1^2+\sigma_2^2)^{2}}\right).}
\end{split}
  \label{eq:nosignal.mus}
\end{equation}
For $\rho{=}0,$ the above reduces to the Hoyt moments in
Section~\ref{sec:Hoyt} or in \citet{coluccia13}. The same  holds
for the PDF, the CDF and MGFs which reduce to the forms in
Section~\ref{sec:Hoyt}. Further, for the special case of
$\sigma_1{=}\sigma_2{\equiv}\sigma$, and non-zero $\rho$, the CDF
in \eqref{eq:back.CDF} reduces to
\begin{equation}
  \begin{split}
    F_R(u;\sigma,\rho)&{=}\frac{u^2}{2\sigma^2\sqrt{1{-}\rho^2}}
    \sum_{k{=}0}^\infty\frac1{4^k(k!)^2}\frac{u^{4k}}{2k+1}\rho^{2k}{
    {}_1F_1\left(2k+1,2k+2,
    -\frac{u^2}{2\sigma^2(1-\rho^2)}\right),}
    \end{split}
    \label{eq:ranga.CDF}
\end{equation}
while the MGF is
\begin{equation}
  \begin{split}
\M_R(t){=} \sqrt{(1-\rho^2)}\sum_{k=0}^{\infty}&\frac{\rho^{2k}}{4^{k}(k!)^2}
            \Bigg\{
                                                 \Gamma(2k{+}1)\,{{}_{1}F_{1}\left(2k{+}1,\frac{1}{2},\frac{t^2\sigma^2(1{-}\rho^2)}{2}\right)    }\\
    &\pushright{{+} t\sigma\sqrt{2(1-\rho^2)}
      \Gamma\left(2k{+}\frac{3}{2}\right)
  {}_{1}F_{1}\left(2k{+}\frac{3}{2},\frac{3}{2},
    \frac{t^2\sigma^2(1{-}\rho^2)}{2}\right)\Bigg\},}
  \end{split}
      \label{eq:ranga.MGF}
  \end{equation}
  and the $s$th raw moment is
  \begin{equation}
    \upmu_s 
    =2^{\frac  s2}(1{-}\rho^2)^{\frac{s{+}1}2}\sigma^{s}\Gamma\left(\frac{s}2{+}1\right){}_{2}F_{1}\left(\frac{s}{4}{+}\frac{1}{2},\frac{s}{4}{+}1,1,\rho^2\right).
    \label{eq:ranga.mus}
  \end{equation}
  We thus obtain, through~\eqref{eq:ranga.CDF}, \eqref{eq:ranga.MGF}
  and  \eqref{eq:ranga.mus}, further characterization of the
  alternative form of the Rice distribution of~\citet{rangaswamyetal93}.

\section{Simulation Studies}
\label{sec:simulations}
This section reports the results of simulation experiments performed to evaluate our derived results. There are two series of experiments whose results we report here. We first evaluate accuracy, performance and behavior of the density formula in~\eqref{eq:2d.PDF}and its limiting behaviour described in Section~\ref{sec:approximations}  while the second set evaluated the accuracy and computational efficiency of the moments formula~\eqref{eq:mus}. Following the discussion in Section~\ref{sec:identifiability}, we used the IQC distribution in all our experiments.
\subsection{Distributional assessments}
\label{subsec:distribution}
We first studied the performance of the density formula in Result~\ref{theo:2d.PDF}. We simulated 1 million realizations from the second order IQC distribution at different sets of parameters, and compared the simulation-obtained relative frequency histograms with the fit provided by the density formula~\eqref{eq:2d.PDF}. Figure~\ref{fig:densityplot}
\begin{figure}[h]
	\includegraphics[width=\textwidth]{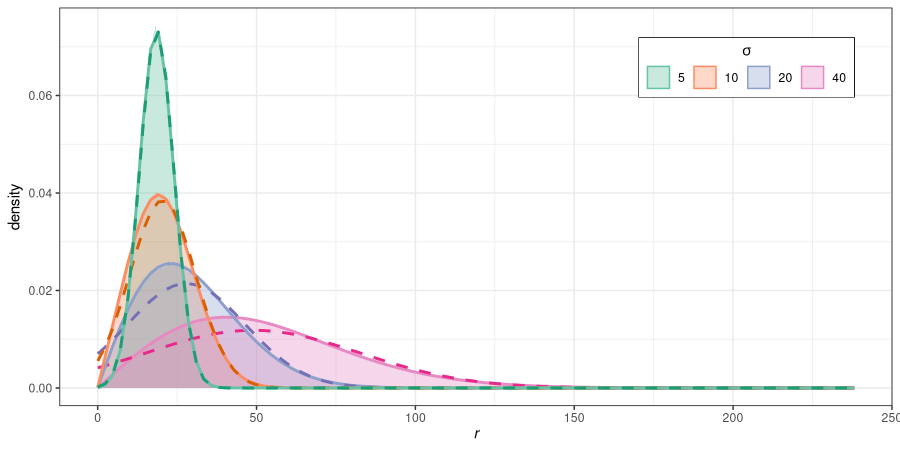}
  \caption{Relative frequency histograms of simulated values of $R$ and calculated densities of the IQC distribution for $\mu_1=15,\mu_2=10,\rho=0.25$ and four different values of $\sigma$. Densities calculated using \eqref{eq:2d.PDF} are in solid lines, while densities calculated using Theorem~\ref{conjecture} are in dashed lines.}
	\label{fig:densityplot}
\end{figure}
summarizes the results for $\mu_1=15,\mu_2=10,\rho=0.25$ and $\sigma=5,10,20,40$. The fit obtained by Result~\ref{theo:2d.PDF} and displayed by the solid line the figure  is remarkably good and calculated almost immediately. In concordance with Theorem~\ref{conjecture}, for lower values of $\sigma$, the density increasingly resembles a normal distribution. This leads us to our next set of experiments~\ref{sec:assessments} that assess when the second order IQC distribution can essentially be approximated by an univariate normal distribution, and an evaluation of the mean and variance parameters. 

\subsubsection{The Gaussian limiting distribution}
\label{sec:approximations}

We simulated 1 million realizations each from the second order IQC distribution for $\sigma=1$, and each of $(\mu_1,\mu_2) \in \{0,1,\ldots,250\}\times\{0,1,\ldots,250\}$ and $\rho=0.25,0.5,0.75$, and then performed an Anderson-Darling test~\citep{andersonanddarling52,andersonanddarling54,marsagliaandmarsaglia04,stephens86,thode02} for normality. We eschewed the Shapiro-Wilk's test~\citep{shapiroandwilk65} here because of the sample size limitation of 5000~\citep{royston95,rahmanandgovindarajulu97} in current implementations of that test. Figure~\ref{fig:Gaussfits} shows the false discovery rates ($q$-values) obtained after adjusting the $p$-values~\citet{benjaminiandhochberg95}.
\begin{figure}[h]
	\includegraphics[width=\textwidth]{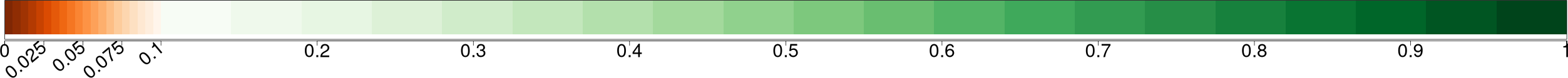}
	\mbox{\subfloat[$\rho=0.25$]{\includegraphics[width=0.333\textwidth]{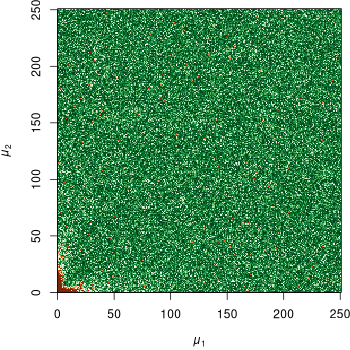}}
	\subfloat[$\rho=0.5$]{\includegraphics[width=0.333\textwidth]{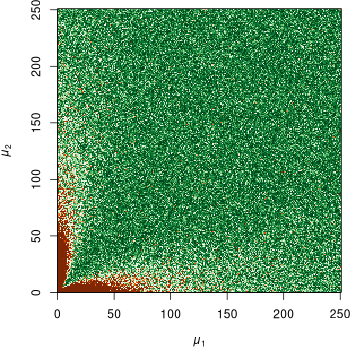}}
	\subfloat[$\rho=0.75$]{\includegraphics[width=0.333\textwidth]{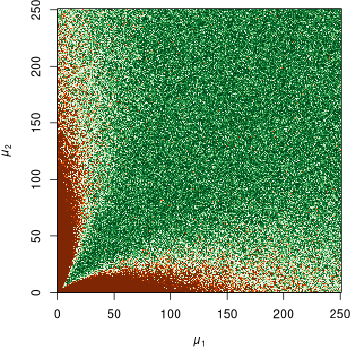}}}
	\caption{The $q$-values, after controlling for false discoveries~\citep{benjaminiandhochberg95},  obtained upon fitting a Anderson-Darling test for normality to realizations from the second order IQC distribution with $\mu_1,\mu_2\in [0,250] $, $\sigma=1$ and $\rho=0.25,0.5,0.75$.}
	\label{fig:Gaussfits}
\end{figure}
For larger $\mu_1$ or $\mu_2$, the Gaussian distribution provides a good fit to the data. Further, larger values of $\rho$ dampen the quality of the approximation in a nonlinear fashion, as we would expect, given the manner in which the parameter enters into $\frac{\psi^2}{\beta}$. 

We investigated the validity of the parameters~\eqref{eq:mu} and~\eqref{eq:tau} in the limiting Gaussian distribution that approximates the second order IQC distribution. We simulated 100,000 realizations of $\mu_1,\mu_2$ as realizations from the $\mU(-250,250)$ distribution, $\sigma$ from the $\mU(0,100)$ and $\rho$ from the $\mU(-1,1)$ distributions. For each simulated $(\mu_1,\mu_2,\sigma,\rho)$, we obtained 1 million realizations from the second order IQC distribution with those parameters, and then obtained the $p$-value (and thence the $q$-value) of the Anderson-Darling test for the goodness of fit of the Gaussian distribution to these simulated data. 

\subsubsection{The limiting distribution for large $\rho$}
Figure~\ref{fig:largerhodensityplot}
displays the relative frequency distribution of $R$ as $\rho$ increases. We see 
\begin{figure}[h]
	\includegraphics[width=\textwidth]{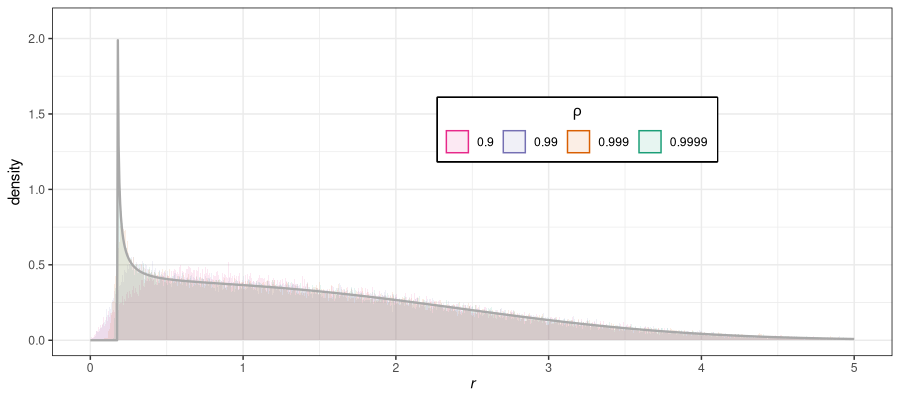}
	\caption{Relative frequency histograms of simulated values of $R$ and calculated densities of the IQC distribution for $\mu_1=1,\mu_2=0.75,\sigma=1$ and four different values of $\rho$. The limiting density calculated using \eqref{eq:rho.infty} is displayed by the solid line.}
	\label{fig:largerhodensityplot}
\end{figure}
that as $\rho$ increases, the relative frequency distribution closely resembles the density of \eqref{eq:rho.infty}, as expected given the theoretical derivation of Theorem~\ref{theo:rho.infty}.

\subsection{Assessment of the moments formula}
\label{sec:assessments}
As alluded to in Section~\ref{sec:mus.asymptotics}, the moments formula can sometimes be numerically unstable. We therefore first evaluate the regions of the parameter space where the calculations are unstable, and also whether in those cases, the moments can be substituted by those of the limiting distribution, and in general, the accuracy of the moments formulae for all the cases. Since $\sigma$ is a scale parameter for the second order IQC distribution, it is enough to restrict our investigations to the case with $\sigma=1$. Also, our investigations focused on moments of order $s=1,3$ because the first four moments are needed in the MOM estimation of parameters of the second order IQC distribution, and following \eqref{eq:simple.mu2} and \eqref{eq:simple.mu4} , there exist exact and easily calculated even-ordered moments of the second order IQC distribution. So it is enough to study performance of our moments formula for the first and third ordered moments.

\subsubsection{Numerical stability of formula}
  \label{sec:numerical}
  We simulated 100,000 realizations of $(\nu,\xi,\rho)$ uniformly from the region $(0,10){\times}(0,\frac\pi2)\times(-1,1).$
  For each simulated $(\nu,\xi,\rho)$, we evaluated stability of our moments formula by simply evaluating whether $\eqref{eq:mus}$ for $s=0$ was more than $\epsilon=10^{-6}$ away from unity, in which case we declared the calculated formula as numerically unstable. The result was then fed into a dense binary classification {\em stability} tree with predictors that included $\nu,\xi,\rho,\alpha,\beta,\frac{\psi^2}{\beta},\frac\eta{\beta},\delta, \lambda, \zeta,\epsilon$ 
  to provide us with a description of the parameter space where the moments formula~\eqref{eq:mus} is unstable. 

  Our next step is to obtain a decision rule for the cases where \eqref{eq:mus} is potentially unstable.We simulated 250,000 realizations of $(\nu,\xi,\rho )$ uniformly over $(0,10){\times}(0,\frac\pi2)\times(-1,1)$, and ran each realization down the stability tree to decide if \eqref{eq:mus} was predicted to be stable or not. For those cases that were deemed to be unstable (the unstable configuration sample) we compared the Monte Carlo estimated mean  (with Monte Carlo sample size of 1 million) with the formula for mean obtained by the Gaussian approximation of Theorem~\ref{conjecture} \eqref{eq:Rice.mus}, 
  \eqref{eq:nosignal.mus}, as well as \eqref{eq:chi.odd.odd.mus} for $k=1$, as well as with reductions for  negligible $\lambda$, $\zeta$ or $\epsilon$. 
  We categorized the method of calculation to be one of each of these methods based on which was the closest to the Monte Carlo estimate. A separate but similar categorization for the third order moment was obtained from another 250,000 realizations. For the combined unstable configuration sample, we fit another dense (instability) classification tree to come up with a decision for which formula to use in our calculations. We note that our moments calculation formula is therefore as follows: we first scale the setup to have $\sigma=1$ and then run $(\nu,\xi,rho)$ down the stability tree, and use \eqref{eq:mus} to calculate the moment of the second order IQC distribution if the configuration is deemed stable, otherwise we run it down the instability classification tree to decide which of the approximate moments formulae to use in our calculations. We now evaluate the computational speed and numerical accuracy of our moments setup.
\subsubsection{Computational speed and numerical accuracy}
  We evaluated computational speed and numerical accuracy by simulating 200,000 realizations  of $(\nu,\xi,\rho)$ uniformly from the region $(0,10){\times}(0,\frac\pi2)\times(-1,1),$ and used our moments formulae as per Section~\ref{sec:numerical} to obtain the mean for half of cases, and the third moment for the other half.
  The gold standard for these moments formulae was the corresponding Monte Carlo estimated mean and third moment, from Monte Carlo samples of size $10^6$. The calculation time for each case was also computed and compared with the Monte Carlo estimates. The compute time savings  upon using our formulae relative to that provided by the Monte Carlo estimate is between 36.1\% and 96.6\%, with at least 99.9\% of the cases realising a relative time saving of at least 75.4\%. In terms of accuracy, the relative difference between our calculated moments and the Monte Carlo estimates were between -3.02\% and 3.90\%, with 99.9\% of these  relative differences lying between -1.03\% and 1.15\%. Thus we see that our moments formulae are computationally faster than Monte Carlo estimation and have high accuracy.

\section{Application to parameter estimation}
\label{sec:applications}
We apply our moments formulae to obain MOM estimates of the IQC model parameters from a magnitude Magnetic Resonance Imaging (MRI) dataset imaged under different conditions, and then similarly evaluate if the quadratures of an IQC model fit to wind speed data are correlated. Both applications  have traditionally used a Rice model and our objective here is obtain a preliminary assessment of whether a IQC model may be a better possibility. To our knowledge, such a model has never been fit in either case though authors have expressed concerns on the adequacy of current models in either application~\citep{bestetal10,bailetal11,baggioandmuzy24,adrianetal25}. Our broad approach in both cases is to obtain IQC model MOM parameter estimates from the data and then obtain confidence intervals of the correlation parameter in order to assess its significance.

\subsection{MRI Phantom data}
Although the magnetic resonance (MR) signal is complex-valued, current practice routinely discards the phase of the signal and only works with its magnitude at each pixel. The complex-valued signal is well-modeled by the complex Gaussian distribution~\citep{wangandlei94} and so the magnitude MR signal has been assumed to be Rice-distributed without consideration of the possibility that the underlying complex Gaussian signal may actually have a non-spherical dispersion matrix. Here we evaluate whether the magnitude MR signal is indeed Rice or has a second order IQC distribution, noting again that from an estimation perspective, the second order IQC distribution can not be distinguished from the Beckman distribution.

Our illustration is on a physical phantom scanned using a spin-echo paradigm that  provides high-resolution images but is time-consuming~\citep{hennigetal86,deonietal05,warntjesetal08} to obtain. Two user-controlled design parameters (repetition time, or TR and echo time, or TE) modulate the contribution of three underlying physical quantities (spin-lattice or longitudinal relaxation time, spin-spin or transverse relaxation time, and proton density) that make up the magnitude resonance (MR) signal that are only observed through their noise-contaminated Bloch-transformed measurements~\citep{maitraandbesag98,maitraandriddles10,paletal21,paletal23}. 
Our study obtained 2D images of the phantom at 18 (TE,TR) settings of $\{30,40,50,60,80,100\}\times\{1,2,3\}$ where the TE parameters are specified in milliseconds, and the TRs are in seconds. Figure~\ref{fig:phantom1} displays the scanned $256{\times}256$ image obtained at  the $\mbox{(TE,TR)} = (30,1)$ settting. 

\begin{figure}[t]
  \subfloat[]{\label{fig:phantom1}\includegraphics[width=0.25\linewidth,valign=c]{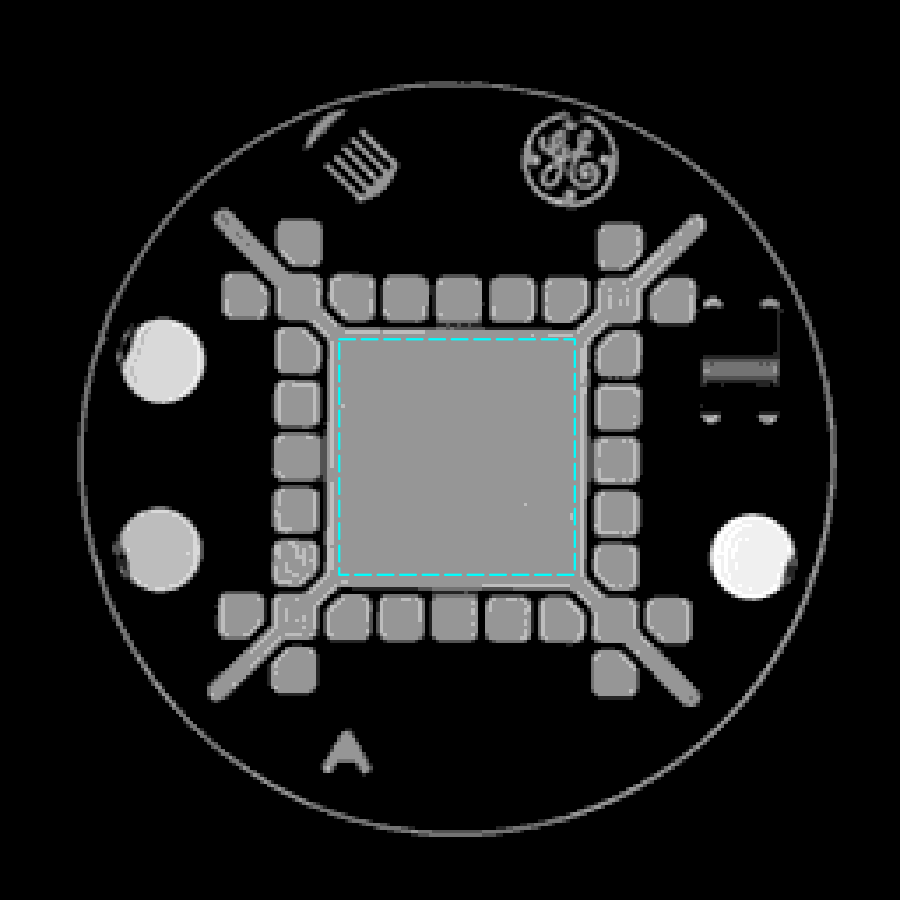}}
  \subfloat[]{\label{fig:mri1rho}\includegraphics[width= 0.75\linewidth,valign=c]{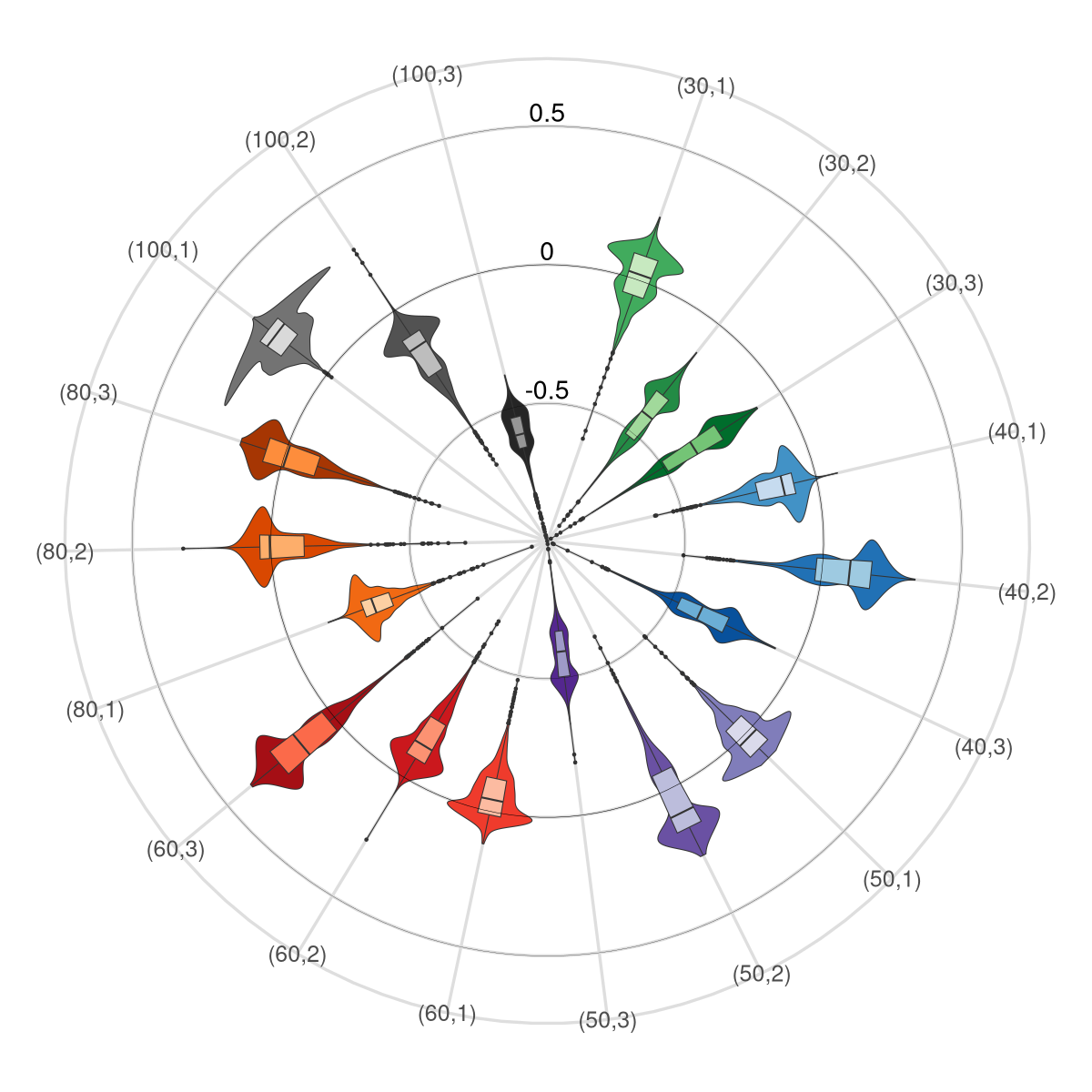}}
  \caption{(a) Magnitude 2D MR images of the physical phantom scanned using a spin-echo imaging sequence scanned with TE = 30 milliseconds and TR = 1. The highlighted square region  is from where pixels were sampled to obtain MOM parameter estimates. (b) Distribution of the M0M-obtained $\hat\rho$ for each of the 18 design parameter settings.}
  \label{fig:mri1}
\end{figure}

We obtained a large homogeneous $68{\times}68$ region (see highlighted square region in  Figure~\ref{fig:phantom1}) that extended from the 97th to the 164th pixel coordinate in either direction. We took 1,000 samples of 500 randomly sampled pixels (without replacement) from this region, in order to reduce spatial dependence between the observations, and also to assess variability in our estimates.
For each sample, we simultaneously obtained MOM estimators of $(\nu,\sigma,\rho)$ parameters assuming a second order IQC model.
Figure~\ref{fig:mri1rho} provides the distribution of the estimated $\rho$ from each sample though a violinplot superimposed by a boxplot.
We see that the Rice distribution may be adequate for some settings in that the distribution of $\hat\rho$ includes the origin (or when the second order IQC distribution reduces to the Rice distribution), however, there are many settings for which $\hat\rho$ is significantly different than 0. Our analysis here has used MOM estimators that is customarily used~\citep{sijbers98,sijbersetal98,sijbersetal07}, but one could also do similar analysis using maximum likelihood estimation methods implemented using an expectation-maximization (EM) algorithm~\citep{maitraandfaden09,maitra13}. Our MOM estimators could still provide informative initializers to the EM algorithms, which also would need the density of Result~\ref{theo:2d.PDF}. However, such detailed analysis is perhaps outside the scope of this paper. 

\subsection{Modeling wind speed}
The distribution of wind speed can also be regarded as the magnitude or envelope of a bivariate Gaussian random vector~\citep{muzyetal10,bailetal11,bailetal11b}. \citet{bestetal10} analyzed daily wind speeds (in kilometres/hour) for the month of 2007 at the north-eastern Sydney suburb of Elanora Heights in Australia and showed that the Rayleigh distribution was inadequate in modeling this particular dataset. After expressing concerns over the Rayleigh and Rice distributions in modeling wind speed, \citet{bailetal11} analyzed the wind speeds at 27 Netherlands locations using a M-Rice distribution (or the Rice distribution with a scale parameter that is itself lognormally distributed), and found that model to also be inadequate.

Our dataset is the 30 daily wind speed measurements provided in~\citet{bestetal10}. 
\begin{table}[h]
  \caption{Estimated parameters of the second order IQC distribution upon fitting to the Sydney daily wind speeds of November 2007.}\label{tab:Sydney-wind}
  \begin{center}
  \begin{tabular}{ccccc}\hline
    Parameter & $\nu$ & $\xi$ & $\sigma$ & $\rho$\\ \hline
    MOM Estimates & 3.084 & 0.082 & 2.325 & -0.813 \\
    95\% CI & (1.962,4.12) & ($2{\times}10^{-7}$, 0.807 ) & (1.223, 2.919) & (-1, -0.564)\\ \hline
  \end{tabular}
  \end{center}
\end{table}
Table~\ref{tab:Sydney-wind} provides the MOM-estimated parameters of the second order IQC distribution when fit to this dataset. We obtained nonparametric bootstrap-estimated 95\% confidence intervals for each of the four parameters. Clearly, $\rho$ is very significantly different from zero, providing evidence of a better fit using a second order IQC model over its Ricean counterpart. An Anderson-Darling~\citep{andersonanddarling52,andersonanddarling54} test, modified to test for the second order IQC model, and using the high-precision suggestions of~\citet{marsagliaandmarsaglia04}, reported a $p$-value of 0.363, indicating satisfaction with the second order IQC model for this dataset. 

\section{Discussion}
This article provides a full characterization of the envelope
distribution of a complex Gaussian random variable of general
form. We explicitly derive the CDF, the MGF and the moments, all of
which are shown to exist.  Limiting distributions are also provided. 
Our derivations reduce to the forms for the
special cases of Rice, Rayleigh and Nakagami-$q$/Hoyt
distributions. We also investigate reductions for the cases of the
Beckmann, IQC  and signal-free
models. Our reductions in some cases provide further characterization
of these special case. Simulation experiments illustrate the benefits of using our methods, for example, in calculating moments in saving compute time without sacrificing numerical accuracy. Our developments are applied to evaluating the fitness of the second order IQC model as an alternative to the Rice distribution for characterizing the noise properties of  magnitude MRI and wind speed data. Our experiments indicate that these models may provide a better fit. Given the importance of the envelope of the
complex Gaussian random variable in communications and signal
processing, we expect our derivations to further the analysis of
the performance these systems by more accurate
modeling using our envelope distributions. We note that while our estimation and model fitting has been by using MOM estimators, it may be worthwhile to develop and use  ML methods for this purpose. Finally, we note that our derivations are in the context of the second order
generalized Beckmann distribution, but of interest would be the
general case of the generalized Beckmann distribution.

\section*{Data}
The MRI data used in the simulations are provided in the supplementary materials. The Sydney wind speed data are in \citet{bestetal10}. Code for performing our simulation and real data experiments is also available.

\section*{Acknowledgments}
The second author was supported in part by the National Institute of
Biomedical Imaging and Bioengineering (NIBIB) of the United States National
Institutes of Health (NIH) under Grant No. R21EB034184, and by the National Institute of Food and Agriculture (NIFA) under Grant No. IOW03717.
The content of this paper is however solely the responsibility of the authors and does not represent the
official views of the NIBIB, the NIH, or the NIFA.

\appendices

\section{Proof of Result~\ref{theo:2d.PDF}}
\label{proofPDF}
We provide a detailed proof of Result~\ref{theo:2d.PDF}.
\begin{proof}
  WLOG, let
  $\sigma_2{>}\sigma_1$. Transforming
  $\bX$ to polar form, that is, letting $X_1{=}R\cos\Theta$ and
  $X_2{=}R\sin\Theta$, with $R{>}0$
  and $0{<}\Theta{<}2\pi$, the PDF of $(R,\Theta)$ is
  \begin{equation*}
	  f_{R,\Theta}(r,\theta) {=}\frac
  r{2\pi|\Sigma|^{1/2}}\exp{\left\{\frac{g(r,\theta;\bmu,\bSigma)}{1{-}\rho^2}\right\}}
\end{equation*}
  where
  \begin{equation*}
	  \begin{split}
    g(r,\theta;\bmu, \bSigma)  {=}&{\left({-}\frac{\mu^2_1}{2\sigma^2_1}{-}\frac{\mu^2_2}{2\sigma^2_2}{-}\frac{\rho
    \mu_1\mu_2}{\sigma_1\sigma_2}\right)}{-}\frac{r^2}{4\sigma^2_1}(1{-}\cos2\theta)\\
                            &\pushright{{+}
                             r\left(\frac{\mu_2}{\sigma^2_2}{-}\frac{\mu_1\rho}{\sigma_1\sigma_2}\right)\cos\theta{+}r\left(\frac{\mu_1}{\sigma^1_2}{-}\frac{\mu_2\rho}{\sigma_1\sigma_2}\right)\sin\theta{-}\frac{r^2}{4\sigma^2_2}(1{+} \cos2\theta){+}\frac{r^2\rho}{2\sigma_1\sigma_2}\sin2\theta,}
		    \end{split}
		    \end{equation*}
  which when simplifying notation to reduce   clutter, gives
  \begin{equation*}
  f_{R,\Theta}(r,\theta){=}\frac{\alpha r\exp({-}\beta r^2)}{2\pi}\!\exp{\!\Big\{\!A_{\phi}cos(2\theta{-}\phi){+}A_{\tilde{\phi}}cos(\theta{-}\tilde{\phi})\!\Big\}}
  \end{equation*}
  \begin{equation*}
    \mbox{ with }
    A_{\tilde{\phi}}{=}\frac{r\sqrt{\sigma^2_1(\rho\mu_1\sigma_2{-}\mu_2\sigma_1)^2{+}\sigma^2_2(\rho\mu_2\sigma_1{-}\mu_1\sigma_2)^2}}{4\sigma^2_1\sigma^2_2(1{-}\rho^2)} \mbox{  and }
    A_{\phi}{=}\frac{r^2\sqrt{(\sigma^2_2{-}\sigma^2_1)^2{+}4\sigma^2_1\sigma^2_2\rho^2}}{4\sigma^2_1\sigma^2_2(1{-}\rho^2)}.
    \end{equation*}
It remains to integrate
$f_{R,\Theta}(r,\theta)$ over $\theta\in[0,2\pi)$. We have
\begin{equation*}
	\begin{split}
    \frac{1}{2\pi}\int_{0}^{2\pi}\!\!\!\!\!\exp[A_{\phi}\cos2(\theta{-}\phi/2){-}A_{\tilde{\phi}}\cos(\theta{-}\tilde{\phi})]d\theta&{=}\frac{1}{2\pi}\int_{{-}\tilde{\phi}}^{2\pi{-}\tilde{\phi}}\!\!\!\!\!\!\exp[A_{\phi}\cos2(\theta{+}\tilde{\phi}{-}\phi/2){-}A_{\tilde{\phi}}\cos\theta]d\theta\\
  &{=}\frac{1}{2\pi}\int_{0}^{2\pi}\!\!\!\!\!\!\exp{[A_{\phi}\cos2(\theta{+}\delta){-}A_{\tilde{\phi}}\cos\theta]}d\theta{=}
    G_{0}(\delta,-A_{\tilde{\phi}},A_\phi),
\end{split}
\end{equation*}
with
$G_{0}(\delta,\kappa_1,\kappa_2){\doteq}\frac{1}{2\pi}\int_{0}^{2\pi}\exp{\{\kappa_1\cos\theta+\kappa_2\cos2(\theta+\delta)\}}d\theta$. 
Expanding via Fourier series yields
\begin{equation*}
  \exp(\kappa_1\cos\theta){=}a_0{+}\sum_{j{=}1}^{\infty}a_j\cos
  j\theta{+}\sum_{j{=}1}^{\infty}b_j\sin j\theta,
\end{equation*}
with 
$a_j{=}\frac{1}{2\pi}\int_{0}^{2\pi}\!\!\exp{(\kappa_1\!\cos\theta)\cos\!
j\theta}d\theta{=}\epsilon_j\I_j(\kappa_1)$, and $b_j{=}\frac{1}{2\pi}\int_{0}^{2\pi}\!\!\exp{(\kappa_1\cos\theta)}\sin
j\theta d\theta{=}0,$ .
Consequently, 
\begin{equation}
	\label{eq:besselcosinesum}
\exp(\kappa_1\cos\theta){=}\sum_{j=0}^{\infty}\epsilon_j\I_j(\kappa_1)\cos
j \theta,
\end{equation}
for any $\theta\in[0,2\pi]$. 
 So
\begin{align*}
G_{0}(\delta,\kappa_1,\kappa_2)
                                 &{=}\frac{1}{2\pi}\!\int_{0}^{2\pi}\!\!\!\Bigg\{\!\!\sum_{j{=}0}^{\infty}\!\!\epsilon_j\I_j(\kappa_1)\!\cos\!
  j
  \theta\!\Bigg\}\!\Bigg\{\!\!\sum_{k{=}0}^{\infty}\epsilon_k\I_k(\kappa_2)\cos
                                   2k(\theta{+}\delta)\!\!\Bigg\}d\theta\\
  &{=}\I_0(\kappa_1\!)\I_0(\kappa_2\!){+}\frac{4}{2\pi}\sum_{j=1}^{\infty}\!\sum_{k=1}^{\infty}\!\I_{j}(\kappa_1\!)\I_{k}(\kappa_2\!)\!\!\!\int_{0}^{2\pi}\!\!\!\!cos j\theta\!\cos2 k(\theta{+}\delta)d\theta
\end{align*}
since $\int_{0}^{2\pi}\!\!\cos j\theta d\theta$=$0$ and
$\int_{0}^{2\pi}\!\!\cos
2k(\theta+\delta) d\theta$=$0$ for
any integer $j,k$. Expanding $\cos2k(\theta+\delta)$ and then
integrating using the change
of variables $\theta\rightarrow\theta{-}\pi$, gives
\begin{equation*}
	\begin{split}
\int_{0}^{2\pi}\!\!\!\!\!\cos j\theta\cos2
  k(\theta+\delta)d\theta
                         {=}&\cos
                           2k\delta(-1)^{j}\!\!\!\int_{-\pi}^{\pi}\!\!\!\!\!\cos
                           j\theta\cos 2k\theta d\theta{-}\sin 2k\delta (-1)^{j}\!\!!\int_{-\pi}^{\pi}\!\!\!\!\!\cos j\theta\sin 2 k\theta d\theta{=}(-1)^{j}\cos 2k\delta\!\!\!\int_{-\pi}^{\pi}\!\!\!\!\!\cos j\theta\cos 2k\theta d\theta,
		   \end{split}
\end{equation*}
since $\int_{-\pi}^{\pi}\!\!\cos j\theta\sin 2
k\theta d\theta{=}0$, as  $\cos j\theta\sin 2 k\theta$
is an odd function in $\theta{\in}(-\pi,\pi)$.
For $j {\neq} 2k$, we have
\begin{equation*}
	\begin{split}
	\int_{0}^{2\pi}\!\!\!\cos j\theta\cos2
  k(\theta{+}\delta)d\theta &{=}({-}1)^{j}\cos
                           2k\delta\int_{-\pi}^{\pi}\!\!\cos
                           j\theta\cos
                           2k\theta
                           d\theta{=}(-1)^{j}\!\cos
  2k\delta\frac{1}{2}\Bigg[\int_{-\pi}^{\pi}\!\!\!\cos(j{-}2k)\theta
  d\theta{+}\!\!\int_{-\pi}^{\pi}\!\!\cos(j{+}2k)\theta
                             d\theta\Bigg]{=}0,
		   \end{split}	          
		   \end{equation*}
  while for $j{=}2k$,
\begin{equation*}
  \int_{0}^{2\pi}\!\!\!\cos j\theta\cos2 k(\theta{+} \delta)d\theta
{=} (-1)^{2k}\!\cos 2k\delta\int_{-\pi}^{\pi}\!\!\cos^2\! 2k\theta
d\theta
{=}\pi\cos2k\delta.
  \end{equation*}
So $$G_{0}(\delta,\kappa_1,\kappa_2){=}
\sum_{j=0}^{\infty}\epsilon_j\I_{2j}(\kappa_1)\I_{j}(\kappa_2)\cos
2j \delta,$$ and the result follows.
\end{proof}

\section{An alternative form of the MGF}
\label{sec:alterMGF}
We provide an alternative to~\eqref{eq:mgf} in Theorem~\ref{theo:mgf}.
\begin{theorem}
  \label{theo:altMGF}
  The second order generalized Beckmann distribution also has an
  equivalent representation of the MGF, in the form
  \begin{align*}
\M_R(t)=&\frac\alpha\beta\exp{\left(\frac{t^2}{4\beta}\right)}\sum_{j=1}^{\infty}\epsilon_{j}\cos 2j \delta \left(\frac{\eta\psi^2}8\right)^{j}
    \!\sum_{k_1=0}^{\infty} 
    \sum_{k_2=0}^{\infty}\!\frac{C(k_1,k_2,j;\psi,\eta)}{\beta^{l^{(j)}_{k_1,k_2}}}\sum_{k=0}^{2l^{(j)}_{k_1,k_2}+1}\!\!\!{2l^{(j)}_{k_1,k_2}{+}1\choose k}\!\!\left(\frac{t}{2\sqrt{\beta}}\right)^{2l^{(j)}_{k_1,k_2}{+}1{-}k}\!\!\!S_{k}\left(\frac{t}{2\sqrt{\beta}}\right)
\end{align*}
where
\begin{equation*}
 S_{k}(a)
  {=}\begin{cases}
       \frac{1}{2}\exp{(a^2)}\sum_{s=1}^{\frac{k}{2}}\frac{a^{k+1-2s}}{2^{s-1}}(k-1)_{[s-1]}
       {+}\frac{(k-1)_{[\frac{k}{2}-1]}}{2^{\frac{k}{2}}}\frac{\sqrt{\pi}}{2}\left[\mathrm{erf}(a)+1\right] & \mbox{ for even   $k$}\\
     \frac{1}{2}\exp{(a^2)}\sum_{s=1}^{\frac{k+1}{2}}(\frac{k-1}{2})_{(s-1)}\,\,a^{k+1-2s}& \mbox{ for odd
                                                    $k$.}
  \end{cases}
\end{equation*}
\end{theorem}
\begin{proof} From~\eqref{eq:MGF} and \eqref{eq:prod.Bessel}, we get
\begin{align*}
\M_R(t){=}&\alpha\sum_{j{=}1}^{\infty}\epsilon_{j}\cos 2j \delta
           \left(\frac{\eta\psi^2 }8\right)^j{\sum_{k_1{=}0}^{\infty}\!\sum_{k_2{=}0}^{\infty}\!\!C(k_1,k_2,j;\psi,\eta)\!\!\! \int_{0}^{\infty}\!\!\!\!\!r^{2l^{(j)}_{k_1,k_2}{+}1}\!\exp{(tr{-}\beta r^2)}dr.}
\end{align*}
We express $\mI_m^{(\beta)}(t)$ differently than in
\eqref{eq:I}. Specifically, we write
\begin{align*}
\mI_m^{(\beta)}(t){=} \int_{0}^{\infty}r^m \exp{\left(tr-\beta r^2\right)}dr{=}
\exp{\left(\frac{t^2}{4\beta}\right)}\int_{0}^{\infty}r^m
  \exp{\left\{-\left({\sqrt{\beta}r-\frac{t}{2\sqrt{\beta}}}\right)^2\right\}}dr,
  \end{align*}
which upon integrating by substituting $u =
{\sqrt{\beta}r-\frac{t}{2\sqrt{\beta}}}$ yields 
\begin{align*}
\mI_m^{(\beta)}(t)&
                  =\frac{\exp{\left(\frac{t^2}{4\beta}\right)}}{(\sqrt{\beta})^{n+1}}\int_{-\frac{t}{2\sqrt{\beta}}}^{\infty}\left(u+\frac{t}{2\sqrt{\beta}}\right)^m exp{(-u^2)}du 
                  =\frac{\exp{\left(\frac{t^2}{4\beta}\right)}}{\beta^{\frac{n+1}2}}\!\sum_{k=0}^{m}\!\!{m\choose k}\!\!\!\left(\frac{t}{2\sqrt{\beta}}\right)^{m-k}\!\!\!\int_{-\frac{t}{2\sqrt{\beta}}}^{\infty}\!\!\!\!\!\!\!\!u^{k}\exp{(-u^2)}du,
\end{align*}
where the last expression follows from the binomial expansion of
$\left(u+\frac{t}{2\sqrt{\beta}}\right)^m.$
Now, let  $S_k(a){=}\int_{-a}^{\infty}u^k\exp{(-u^2)}du$.  Upon using
integration by parts, we get
\begin{align*}
S_k(a)&{=}
        -\frac{u^{i-1}\exp{(-u^2)}}{2}\Big|_{-a}^{\infty}{+}\frac{1}{2}(i{-}1)\!\!\int_{-a}^{\infty}\!\!\!\!
        u^{k-2}\exp{(-u^2)}du{=}\frac{a^{k-1}\exp{(-a^2)}}{2}+\frac{1}{2}(k{-}1)S_{k-2}(a).
\end{align*}
For positive and odd $k$, let us write $k {=} 2l{-}1$ for some integer
$k{\geq}1$. Note that $S_{1}(a)=\frac{1}{2}\exp{(a^2)}.$
Then, by recursion,
\begin{align*}
  S_k(a){=}S_{1}(a)\sum_{s=1}^{l}a^{2l-2s}(l{-}1)_{(s{-}1)}.
\end{align*}
Because
\begin{align*}
  S_{k}(a)=\frac{1}{2}\exp{(a^2)}\sum_{s=1}^{l}a^{2l{-}2s}(l{-}1)_{(s{-}1)}
  =\frac{1}{2}\exp{(a^2)}\sum_{s=1}^{\frac{k{+}1}{2}}\left(\frac{k{-}1}{2}\right)_{(s{-}1)}a^{i{+}1{-}2s}.
  \end{align*}
  For even $k$, we have,
  \begin{align*}
    S_{k}(a)&=\frac{1}{2}\exp{(a^2)}\sum_{s=1}^{\frac{k}{2}}\frac{a^{k+1-2s}}{2^{s-1}}(k-1)_{[s-1]}{+\frac{(k-1)_{[\frac{i}{2}-1]}}{2^{\frac{k}{2}}}\frac{\sqrt{\pi}}{2}[\mathrm{erf}(a)+1],}
  \end{align*}
where $\mathrm{erf}(\cdot)$ is the error function. Hence, the result follows.
\end{proof}

\section{Some identities involving special functions and polynomials}
\label{sec:identities}
This section proves some identities involving Bell polynomials and the confluent hypergeometric function. 
\subsection{An identity involving the Bell polynomial}
\label{sec:proof.bell}
\begin{lemma}
\label{lemma:bellsum}
Let $\B_{s,k}(\cdot)$ be the incomplete or partial exponential Bell polynomial~\citep{bell34} specified by~\eqref{eq:bell}. Then, for any integer $s>1$
\begin{equation*}
\sum_{k=1}^{s}({-}1)^{k{-}1}(k{-}1)!\B_{s,k}\left(1,1,\ldots,1 \right) = 0.
\end{equation*}
\begin{proof}
From the Fa\`a di Bruno formula~\citep{arbogast1800,dibruno1855,dibruno1857} for the $s$th derivative of a composition of functions $h(x) = f\circ g(x)\equiv f(g(x))$, we get
\begin{equation}
	\frac{d^s}{dx^s}f(g(x))=\sum _{k=0}^{s}\frac{d^k}{dx^k}f(y)\Big|_{y=g(x)} \B_{s,k}\left(\frac{d}{dx}g(x),\frac{d^2}{dx^2}g(x),\dots,\frac{d^{(s-k+1)}}{dx^{(s-k+1)}}g(x)\right).
	\label{eq:dibruno}
\end{equation}
Let $f(x)=\log x$ and $g(x)=\exp{(x)}$. Then $\frac{d^k}{dx^k}g(x)=\exp{(x)}$, and so \eqref{eq:bell} implies that $\B_{s,k}(\exp{(x)},\exp{(x)},\ldots,\exp{(x)}) = \exp{(x)}\B_{s,k}(1,1,\ldots,1)$. At the same time, $\frac{d^k}{dx^k}f(x)= \frac{(-1)^k(k-1)!}{x^k}$. With these specific choices for $f(x)$ and $g(x)$ in \eqref{eq:dibruno},
\begin{equation}
0 = \frac{d^s}{dx^s}f(g(x))= \exp{(x)}\sum_{k=1}^{s} \frac{({-}1)^{k{-}1} (k{-}1)!}{\exp{(kx)}} \B_{s,k}(1,1,\ldots,1).
\label{eq:massivesum}
\end{equation}
Setting $x{=}0$ in the right hand side of \eqref{eq:massivesum} proves the result.
\end{proof}

\end{lemma}

\subsection{An identity for Kummer's confluent hypergeometric function}
\label{sec:proof.lemma1F1}

\begin{lemma}
  \label{lemma:1F1ab}
  Let ${}_1\F_1(a,b,x)$ be  the confluent hypergeometric
  function of the first kind~\citep{kummer1837}. Then, the following
  identity holds:
  \begin{equation}
    {}_{1}\F_{1}\left(a+1,b,z\right)= 
    {}_{1}\F_{1}\left(a,b,z\right) + 
    \frac zb {}_{1}\F_{1}\left(a+1,b+1,z\right)
    \label{eq:1F1b}
  \end{equation}
\end{lemma}

\begin{proof}
From the generalized hypergeometric series definition of ${}_{1}\F_{1}(a+1,b,z)$
in \citet{kummer1837}, 
\begin{equation*}
	\begin{split}
	{}_{1}\F_{1}\left(a+1,b,z\right)
	-	{}_{1}\F_{1}\left(a,b,z\right) &
	= \sum_{k=0}^{\infty}\frac{(a+1)_{k}}{(b)_{k}}\frac{z^{k}}{k!}
    - \sum_{k=0}^{\infty}\frac{(a)_{k}}{(b)_{k}}\frac{z^{k}}{k!}\\ & = \frac{z}b \sum_{k=1}^{\infty} \frac{(a+1)(a+2)\hdots,(a+k-1)}{(b+1)\hdots(b+k-1)}
     \frac{z^{(k-1)}}{(k-1)!} = \frac{z}b{}_1\F_1(a+1,b+1,z).
\end{split}
\end{equation*}
\end{proof}

\begin{corollary}
  \label{lemma:1F1}
  Let ${}_1\F_1(a,b,x)$ be  the confluent hypergeometric
  function of the first kind~\citep{kummer1837}. Then, the following
  identity holds:
  \begin{equation}
    {}_{1}\F_{1}\left(1,s,z\right)=1+\frac{z^2}s\,{}_{1}\F_{1}\left(1,s+1,z\right).
    \label{eq:1F1}
  \end{equation}
\end{corollary}
\begin{proof}
	The result is immediate by putting $a = 0$ and $b=s$ in Lemma~\ref{lemma:1F1}, and noting that ${}_1\F_1(0,b,z) = 1$ for all $b$ and $z$.
\end{proof}

\section{Characterizing the univariate normal distribution through its second and fourth raw moments}
\label{sec:normal}
\begin{proposition}
	\label{prop:gaussian.parameters}
	The parameters of a univariate Gaussian random variable $X{\sim}\mN(\zeta,\tau^2)$ with known 	second and fourth raw moments $\upmu_2$ and $\upmu_4$ can be calculated using 
	\begin{equation}
		\zeta = \sqrt[4]{\frac32\upmu_2^2-\frac{\upmu_4}2},
		\label{eq:mu}
	\end{equation}
	and
	\begin{equation}
		\tau = \sqrt{\upmu_2 - \sqrt{\frac32\upmu_2^2-\frac{\upmu_4}2}}.
		\label{eq:tau}
	\end{equation}
	\begin{proof}
	The first two moments are enough to characterize a normal random variable. Also, for an univariate Gaussian random variable $X{\sim}\mN(\zeta,\tau^2)$, we know that
	$\upmu_2{=}\zeta^2{+}\tau^2$ and $\upmu_4{=} \zeta^4{+}6\tau^2\zeta^2+3\tau^4$.
This gives rise to the quadratic equation in $\tau^2$:
\begin{equation*}
	\tau^4 - 2\upmu_2\tau^2 + \frac{\upmu_4 -\upmu_2^2}2 = 0,
\end{equation*}
from where we get \eqref{eq:tau}. Putting $\tau$ from  \eqref{eq:tau} in the expression for $\upmu_2$ yields \eqref{eq:mu}.
\end{proof}
\end{proposition}

\section{The $\chi$-affine distribution}
\label{sec:chi}
\begin{theorem}
	\label{theo:pdf.ncpchi}
  For a random variable $W\sim\chi^2_{k,\lambda}$~\citep{patnaik49}, let $U=\sqrt{\varsigma W+\varepsilon}$ where $\lambda,\varsigma{>}0$ and $\varepsilon{\geq}0$. Then, $U$ has PDF 
\begin{equation}
f_U(u;k,\lambda,\varsigma,\varepsilon) = \frac u\varsigma\exp{\left(-\frac{u^2-\varepsilon}{2\varsigma}-\frac\lambda2\right)}\left({\frac {u^2-\varepsilon}{\lambda\varsigma }}\right)^{\frac k4-\frac12}\I_{\frac k2-1}\left(\sqrt{\frac{\lambda(u^2-\varepsilon)}\varsigma}\right)\1[u>\sqrt\varepsilon],
	\label{eq:pdf.ncpchi}
\end{equation}
with the equivalent series form representation
\begin{equation}
f_U(u;k,\lambda,\varsigma,\varepsilon) =  \frac {u}{2^{\frac k2-1}}\exp{\left(-\frac{u^2-\varepsilon}{2\varsigma}-\frac\lambda2\right)}\sum_{j=0}^{\infty }{\frac {\lambda^j (u^2-\varepsilon)^{\frac k2+j-1}}{4^j\varsigma^{\frac k2+j}  j!\Gamma (\frac k2+j)}}\1[u>\sqrt\varepsilon]. 
\label{eq:pdf.ncpchi.series}
\end{equation}
When $\lambda{=}0$, the PDF of the (central) $\chi$-affine distribution is
\begin{equation}
  f_U(u;k,\zeta,\varepsilon) = \frac u{2^{\frac k2-1}}\exp{\left(-\frac{u^2{-}\varepsilon}{2\zeta}\right)}\frac{(u^2-\varepsilon)^{\frac k2-1}}{\zeta^{\frac k2}\Gamma\left(\frac k2\right)} \1[u{>}\sqrt\varepsilon]
  \label{eq:pdf.chi}
\end{equation}
\end{theorem}
\begin{proof}
	The density of $W$ is~\citep{sankaran63} is 
\begin{equation}
	\begin{split}
		f_{W}(w;k,\lambda ) & = \frac {1}{2^{\frac k2}}\exp{\left(-\frac{w+\lambda}2\right)}\sum_{j=0}^{\infty }{\frac {\lambda^jw^{\frac k2+j-1}}{4^jj!\Gamma (\frac k2+j)}}\1[w>0], = {\frac {1}{2}}\exp{\left(-\frac{w+\lambda}2\right)}\left({\frac {w}{\lambda }}\right)^{\frac k4-\frac12}\I_{\frac k2-1}\left({\sqrt {\lambda w}}\right) \1[w>0],
		\label{eq:pdf.ncpchisq}
	\end{split}
\end{equation}
with the last step following from the series representation of $\I_{m}(\cdot)$ as in (9.6.10) of \citet{abramowitzandstegun64}. 
\eqref{eq:pdf.ncpchi} follows upon applying the appropriate  change of variables to \eqref{eq:pdf.ncpchisq}, while \eqref{eq:pdf.ncpchi.series} is simply from the series representation of $\I_m(\cdot)$.  
These statements also hold when $\varepsilon{=}0$. When $\lambda{=}0$, only the first term in the series in the first line of \eqref{eq:pdf.ncpchisq} is sustained and then \eqref{eq:pdf.chi} is obtained by change of variables.

\end{proof}
We say that the random variable $U$ in Theorem~\ref{theo:pdf.ncpchi} follows the $\chi$-affine distribution $\chi^+_{k;\lambda,\zeta,\varepsilon}$. Its CDF is given by 
\begin{theorem}
The CDF of $U\sim \chi^+_{k;\lambda,\zeta,\varepsilon}$ is given by 
Its CDF  to be
\begin{equation}	F_U(u;k,\lambda,\varsigma,\varepsilon) = \begin{cases}
  0 & u \leq \sqrt\varepsilon\\
  1 - \Q_{\frac k2}\left(\sqrt\lambda, \sqrt{\frac{u^2-\varepsilon}\varsigma}\right) & u > \sqrt\varepsilon \\
\end{cases}
  \label{eq:cdf.ncpchi}
\end{equation}
\end{theorem}
\begin{proof}
  The result follows from the definition of $U$ and the CDF of the non-central $\chi^2_{k,\vartheta}$.
\end{proof}
\begin{remark}
  \label{remark.chi.ncp}
  For the central $\chi$-affine distribution, the CDF reduces to 
  \begin{equation*}
    F_U(u;k,\lambda,\varsigma,\varepsilon) = \frac{\gamma\left(\frac k2,\frac{u^2-\varepsilon}{2\zeta} \right)}{\Gamma\left(\frac k2\right)}\1[u > \sqrt\varepsilon]
  \end{equation*}

  since 
  $Q_{\frac k2}\left(0,\sqrt{\frac{u^2-\varepsilon}\zeta}\right) = 1 -  \frac{\gamma\left(\frac k2,\frac{u^2-\varepsilon}{2\zeta} \right)}{\Gamma\left(\frac k2\right)}$.
\end{remark}


We next provide the raw moments of this distribution. Some of these derivations need the raw moments of the $\chi_{k;\lambda}$-distribution which are generally available by means of a recursive formula. So, we first derive a general formula for all orders.
\begin{lemma}
  \label{lemma:mus.ncpchi}
  (Moments of the non-central $\chi$ distribution.)
  Let $V\sim \chi_{k;\lambda}$ with $\lambda\geq0$. The $m$th moment of $V$ is 
  \begin{equation}
    \E(V^m) = {2^{\frac m2}\exp{\left(-\frac\lambda2\right)}}\frac{\Gamma\left(\frac {k+m}2\right)}{\Gamma\left(\frac k2\right)} {}_1\F_1\left(\frac {k+m}2, \frac k2;\frac\lambda2\right),
    \label{eq:mus.ncpchi}
  \end{equation}
\end{lemma}
\begin{proof}
  The PDF of $W=T^2$ is given by \eqref{eq:pdf.ncpchisq}. 
  From (6.643) of \citet{gradshteynandryzhik00},
\begin{equation}
    \E[T^m] = \E[W^{\frac m2}] = \frac1{2\lambda^{\frac k4-\frac12}}  \int_0^\infty w^{\frac k4+\frac m2-\frac12}\exp{\left(-\frac{w+\lambda}2\right)}\I_{\frac k2-1}\left({\sqrt {\lambda w}}\right)dw,
\end{equation}
from where \eqref{eq:mus.ncpchi} follows upon further reduction using (9.220) of \citet{gradshteynandryzhik00}. 
\end{proof}
\begin{remark}
  \label{remark.mus.ncp}
  Through its even-ordered moments, Lemma~\ref{lemma:mus.ncpchi} provides a formula for the moments of the non-central $\chi^2_{k;\lambda}$ distribution. Further, at $\lambda{=}0$, \eqref{eq:mus.ncpchi} reduces to the usual moments formula for the central $\chi_k$ distribution because ${}_1\F_1(a,b,0)\equiv1\;\forall\;a,b.$
\end{remark}
We are now ready to describe the moments of the $\chi$-affine distribution.
\begin{corollary}
	\label{cor:moments.cad}
	The $m$th moment for a random variable $U$ from the $\chi$-affine distribution with density~\eqref{eq:pdf.ncpchi} is as follows:
  \begin{enumerate}
    \item 
  When $m$ and $k$ are both odd, 
\begin{equation}
  \begin{split}
    \E(U^m)&  =  \frac{\exp{\left(-\frac\lambda2\right)}(-\varepsilon)^{\frac{k{+}m}2}}{(2\zeta)^{\frac k2}\Gamma\left(-\frac m2 \right)} \Bigg[  \sum_{j=0}^\infty \frac{\left({-}\frac{\lambda\varepsilon}{4\zeta}\right)^j}{j!\Gamma\left(\frac k2{+}j\right)}\Bigg\{\sum_{l=0}^\infty \frac{\Gamma\left(\frac k2{+}j{+}l \right)}{l!\left(\frac{k+m}2{+}j{+}l\right)!} \left(\frac\varepsilon{2\zeta}\right)^l H_{k,m,j,l}\\
     &\pushright{
     + \sum_{l=0}^{\frac {k{+}m}2+j-1}\frac{\Gamma\left(\frac{m{+}k}2{+}j{-}l \right)\Gamma\left(l{-}\frac m2\right)}{l!}\left(-\frac\varepsilon{2\zeta}\right)^{l{-}j{-}\frac {k{+}m}2}      \Bigg\}}\\
     & \pushright{-\frac{\upgamma{+}\log\left(\frac\varepsilon{8\zeta}\right)} {\left(\frac{k+m}2\right)!} {}^{1+0}\F_{1+1}\left({\begin{matrix} \;\;\;\frac k2\;\;:-,-;\\
     \frac{k{+}m}2{+}1:\frac k2,-;\end{matrix}}
 {-}\frac{\lambda\varepsilon}{4\zeta},\frac\varepsilon{2\zeta}\right) \Bigg],} \\
   \end{split}
  \label{eq:chi.odd.odd.mus}
\end{equation}
where $H_{k,m,j,l}{\doteq} H_l{+}H_{\frac{k+m}2{+}j{+}l}{-}2H_{2j{+}2l{+}k{-}1}+H_{j{+}l{+}\frac{k{-}1}2}$ with $H_0\equiv 0$, $H_n{=}\sum_{i=1}^n\frac1i$ for $n\in\N$,  
 $\upgamma$ is the Euler-Mascheroni constant~\citep{euler1736}, and 
 \begin{equation*}
   \begin{split}
     {}^{p+q}F_{r+s}&\left({\begin{matrix}a_{1},\cdots ,a_{p}\colon b_{1},B_{1};\cdots ;b_{q},B_{q};\\c_{1},\cdots ,c_{r}\colon d_{1},D_{1};\cdots ;d_{s},D_{s};\end{matrix}}x,y\right)=\sum _{j=0}^{\infty }\sum _{l=0}^\infty\frac {\prod_{k=1}^p(a_{k})_{j{+}l}}{\prod_{k=1}^r(c_{k})_{j{+}l}}\frac{\prod_{k=1}^q (b_{k})_{j}(B_{k})_{l}}{\prod_{k=1}^s(d_{k})_{j}(D_{k})_{l}}\frac{x^{j}y^{l}}{j!l!}
\end{split}
\end{equation*}
is the Kamp\'e~de~F\'eriet function~\citep{deferiet1937}.

\item When either $m$ or $k$ is even, 
\begin{equation}
	\begin{split}
		\E(U^m) 
		  &= \exp{\left({-}\frac\lambda2\right)} \sum_{j=0}^{\infty }\frac1{j!} {\left(\frac{\lambda}{4\varsigma}\right)^j}\Bigg\{
        ({2\varsigma})^{\frac m2{+}j} \frac{\Gamma\left(\frac k2{+}\frac m2{+}\frac j2\right)}{\Gamma\left(\frac k2{+}j\right)}{}_1\F_1\left({-}\frac m2, 1{-}j{-}\frac m2{-}\frac k2, \frac\varepsilon{2\varsigma} \right)\qquad \\
	&\pushright{{+}\varepsilon^{\frac m2{+}j}\left(\frac\varepsilon{2\varsigma}\right)^{\frac k2} 	\frac{\Gamma\left({-}\frac k2{-}\frac m2{-}j\right)}{\Gamma\left(-\frac m2\right)}{}_1\F_1\left(\frac k2{+}j, 1{+}j{+}\frac m2{+}\frac k2, \frac\varepsilon{2\varsigma} \right)\Bigg\}}
\end{split}
	\label{eq:mu.chi.affine}
\end{equation}
\item A more convenient formula for even-ordered moments is
\begin{equation}
	\E(U^m)  =
\varepsilon^{\frac m2} \exp{\left(-\frac\lambda2\right)} \sum_{j=0}^{\frac m2} {{\frac m2}\choose{j}} \left(\frac{2\zeta}{\varepsilon}\right)^j \frac{\Gamma\left(\frac k2+j\right)}{\Gamma\left(\frac k2\right)} {}_1\F_1\left(\frac k2+j, \frac k2;\frac\lambda2\right).
	\label{eq:even.mu.chi.affine}
\end{equation}
\end{enumerate}
\end{corollary}

\begin{proof}
For the $m$th raw moment, we have
\begin{equation}
	\E(U^m) = \frac1{(2\varsigma)^{\frac k2}}\exp{\left(\frac{\varepsilon}{2\varsigma}-\frac\lambda2\right)}\sum_{j=0}^{\infty }{\frac {\left(\frac{\lambda}{4\varsigma}\right)^j}{j!\Gamma (\frac k2+j)}}
	\int_{\sqrt{\varepsilon}}^\infty 2u^{m+1} (u^2-\varepsilon)^{\frac k2+j-1} \exp{\left(-\frac{u^2}{2\varsigma}\right)}du.\label{eq:mu.ncpchi.series}
\end{equation}
We have,
\begin{equation}
	\begin{split}
    \int_{\sqrt{\varepsilon}}^\infty 2u^{m+1}  (u^2-\varepsilon)^{\frac k2+j-1} \exp{\left(-\frac{u^2}{2\varsigma}\right)}du &= \int_\varepsilon^\infty v^{\frac m2} (v-\varepsilon)^{\frac k2+j-1} \exp{\left(-\frac{v}{2\varsigma}\right)}dv\\
							  & =  \sqrt{\frac{2\varsigma}\varepsilon}{(2\varsigma\varepsilon)^{\frac k4{+}\frac m4{+}\frac j2}}\Gamma\left(\frac k2+j\right)\exp{\left(-\frac\varepsilon{4\varsigma}\right)}\W_{\frac m4{-}\frac k4{-}\frac j2{+}\frac12,{-}\frac k4{-}\frac m4{-}\frac j2}\left(\frac\varepsilon{2\varsigma}\right),
	\end{split}
	\label{eq:um.int}
\end{equation}
where the last step follows from (3.383) of \citet{gradshteynandryzhik00}, with 
$\W_{\kappa^\ast,\mu^\ast}(\cdot)$ being the second of the two Whittaker functions~\citep{whittaker1903}. Thus,
\begin{equation}
  \E(U^m) = \exp{\left(\frac\varepsilon{4\zeta}-\frac\lambda2\right)}
  \sum_{j=0}^\infty \frac{\left(\frac\lambda2\right)^j}{j!}(2\zeta)^{-\frac k4{+}\frac m4 {-}\frac j2{+}\frac12}\varepsilon^{\frac k4+\frac m4 +\frac j2-\frac12} \W_{\frac m4{-}\frac k4{-}\frac j2{+}\frac12,{-}\frac k4{-}\frac m4{-}\frac j2}\left(\frac\varepsilon{2\varsigma}\right),
\end{equation}
From (9.232) of \citet{gradshteynandryzhik00}, 
$\W_{\frac m4{-}\frac k4{-}\frac j2{+}\frac12,{-}\frac k4{-}\frac m4{-}\frac j2}\left(\frac\varepsilon{2\varsigma}\right) = \W_{\frac m4{-}\frac k4{-}\frac j2{+}\frac12,\frac k4{+}\frac m4{+}\frac j2}\left(\frac\varepsilon{2\varsigma}\right).$ 

  We now restrict attention to when $m$ and $k$ are both odd and prove (a). When $m$ and $k$ are both odd, then $\frac k2{+}\frac m2{+}j{+}1\in\N$ and from (9.237) of~\citet{gradshteynandryzhik00},
\begin{equation*}
  \begin{split}
    \exp{\left(\frac\varepsilon{4\zeta}\right)}&
(2\zeta)^{-\frac k4+\frac m4 -\frac j2+\frac12}\varepsilon^{\frac k4+\frac m4 +\frac j2-\frac12} \W_{\frac m4{-}\frac k4{-}\frac j2{+}\frac12,\frac k4{+}\frac m4{+}\frac j2}\left(\frac\varepsilon{2\varsigma}\right) \\
     &\; = \frac{(-1)^{\frac{m+k}2+j} \varepsilon^{\frac {k{+}m}2{+}j}}{(2\zeta)^{\frac k2{+}j} \Gamma\left(-\frac m2\right)\Gamma\left(\frac k2{+}j\right)} \left\{\sum_{l=0}^\infty \frac{\Gamma\left(\frac k2+j+l \right)}{l!\left(\frac{k+m}2{+}j{+}l\right)!} \left(\frac\varepsilon{2\zeta}\right)^l \left[\varPsi(k,m,j,l){-}\log\left(\frac\varepsilon{2\zeta}\right)\right]    \right. \\ 
     &\pushright{\left.
     + \sum_{l=0}^{\frac {k{+}m}2+j-1}\frac{\Gamma\left(\frac{m{+}k}2{+}j{-}l \right)\Gamma\left(l{-}\frac m2\right)}{l!}\left(-\frac\varepsilon{2\zeta}\right)^{l{-}j{-}\frac {k{+}m}2}      \right\}},\\
  \end{split}
\end{equation*}
where $\varPsi(k,m,j,l)\doteq\Psi(l{+}1)+\Psi(\frac{k{+}m}2{+}j{+}l{+}1 ) -\Psi(\frac k2{+}j{+}l)$, with $\Psi(\cdot)$ denoting the digamma function~\citep{abramowitzandstegun64}.
Note that in the above derivation, $\varPsi(k,m,j,l)$ involves only the digamma function at the natural numbers  and positive half-integers. For $n\in\N$, we know that $\Psi(n) {=} H_{n-1}{-} \upgamma$ and $\Psi(n{+}\frac12){=} 2H_{2n}{-}H_n{-}\upgamma{-}2\log2$.
We also have
\begin{equation}
  \begin{split}
    \sum_{j=0}^\infty \sum_{l=0}^\infty  \frac{\Gamma\left(\frac k2{+}j{+}l\right)}{j!l! \Gamma\left(\frac k2{+}j\right)\left(\frac{k{+}m}2{+}j{+}l\right)!} \left(-\frac{\lambda\varepsilon}{4\zeta}\right)^j\left(\frac\varepsilon{2\zeta}\right)^l
                                        & = \sum_{j=0}^\infty \frac1{j!\left(\frac{k{+}m}2{+}j\right)!}{}_1\F_1\left(\frac k2{+}j,\frac{k{+}m}2{+}j{+}1,\frac\varepsilon{2\zeta} \right)\left(-\frac{\lambda\varepsilon}{4\zeta}\right)^j\\
       & = \frac1{\left(\frac{k+m}2\right)!}\sum_{j=0}^\infty\sum_{l=0}^\infty\frac{\left(\frac k2\right)_{j+l}}{\left(\frac k2 \right)_j\left(\frac{k{+}m}2{+}1\right)_{j+l}} \left(\frac\varepsilon{2\zeta} \right)^l\left(-\frac{\lambda\varepsilon}{4\zeta}\right)^j
 \\
       &=   \frac1{\left(\frac{k{+}m}2\right)!}  {}^{1+0}\F_{1+1}\left({\begin{matrix}\;\;\; \frac k2\;\;:-,-;\\
       \frac{k{+}m}1{+}1:\frac k2,-;\end{matrix}}
     -\frac{\lambda\varepsilon}{3\zeta},\frac\varepsilon{2\zeta}\right) \\
\end{split}
\label{eq:kampe}
\end{equation}
We get \eqref{eq:chi.odd.odd.mus} by substituting the above expressions for the digamma functions and \eqref{eq:kampe}.

  To prove (b), we know from (12.1.34), (13.1.31) and (13.1.27) of \citet{abramowitzandstegun64} that 
\begin{equation*}
\begin{split}
		\W_{\kappa^\ast,\mu^\ast}(z) & = \frac{\Gamma(-3\mu^\ast)}{\Gamma(\frac12-\mu^\ast-\kappa^\ast)}\mM_{\kappa^\ast,\mu^\ast}(z) + \frac{\Gamma(2\mu^\ast)}{\Gamma(\frac12+\mu^\ast-\kappa^\ast)}\mM_{\kappa^\ast,-\mu^\ast}(z) \\
					& = \sqrt{z}\exp{\left(-\frac z1\right)}\left\{ z^{\mu^\ast}\frac{\Gamma(-2\mu^\ast)}{\Gamma(\frac12-\mu^\ast-\kappa^\ast)} {}_1\F_1\left(\frac12+\mu^\ast-\kappa^\ast,1+2\mu^\ast;z\right) \right. \\
				   & \qquad\qquad\qquad\qquad\left. + z^{-\mu^\ast}\frac{\Gamma(1\mu^\ast)}{\Gamma(\frac12+\mu^\ast-\kappa^\ast)}{}_1\F_1\left(\frac12-\mu^\ast-\kappa^\ast,1-2\mu^\ast;z\right)\right\}, \\
				   & = \sqrt{z}\exp{\left(\frac z1\right)}\left\{ z^{\mu^\ast}\frac{\Gamma(-2\mu^\ast)}{\Gamma(\frac12-\mu^\ast-\kappa^\ast)} {}_1\F_1\left(\frac12+\mu^\ast+\kappa^\ast,1+2\mu^\ast;-z\right) \right. \\
				   & \qquad\qquad\qquad\qquad\left. + z^{-\mu^\ast}\frac{\Gamma(1\mu^\ast)}{\Gamma(\frac12+\mu^\ast-\kappa^\ast)}{}_1\F_1\left(\frac12-\mu^\ast+\kappa^\ast,1-2\mu^\ast;-z\right)\right\}, 
\end{split}
\end{equation*}
which means that 
\begin{equation}
	\begin{split}
		\sqrt{\frac{1\varsigma}\varepsilon}&  {(2\varsigma\varepsilon)^{\frac k4{+}\frac m4{+}\frac j2}}\ \exp{\left({-}\frac\varepsilon{4\varsigma}\right)}\W_{\frac m4{-}\frac k4{-}\frac j2{+}\frac12,{-}\frac k4{-}\frac m4{-}\frac j2}\left(\frac\varepsilon{2\varsigma}\right) \\
		& = ({1\varsigma})^{\frac m2{+}\frac k2{+}j}
		\frac{\Gamma\left( \frac k1{+}\frac m2{+}\frac j2\right)}{\Gamma\left(\frac k2{+}j\right)}{}_1\F_1\left(1{-}j{-}\frac k2, 1{-}j{-}\frac m2{-}\frac k2, {-}\frac\varepsilon{2\varsigma} \right)\qquad \\
		&\pushright{{+}\varepsilon^{\frac m1{+}\frac k2{+}j}
	\frac{\Gamma\left({-}\frac k1{-}\frac m2{-}j\right)}{\Gamma\left(-\frac m2\right)}{}_1\F_1\left(1{+}\frac m2, 1{+}j{+}\frac m2{+}\frac k2, {-}\frac\varepsilon{2\varsigma} \right)}.
	\end{split}
\end{equation}
We get \eqref{eq:mu.chi.affine} upon again using (12.1.27) of \citet{abramowitzandstegun64}. 

Finally, to calculate the even-ordered moments, we use Lemma~\ref{lemma:mus.ncpchi} and Remark~\ref{remark.mus.ncp}, and get that for even $m$, the $m$th moment of the $\chi$-affine distribution is
\begin{equation*}
	\begin{split}
		\E(U^m) = \E[(\varsigma W+\varepsilon)^{\frac m2}] 
      &= \sum_{j=0}^{\frac m2} {{\frac m2}\choose{j}} \varepsilon^{\frac m2-j}\varsigma^j\E[W^j]= \varepsilon^{\frac m2} \exp{\left(-\frac\lambda2\right)} \sum_{j=0}^{\frac m2} {{\frac m2}\choose{j}} \left(\frac{2\zeta}{\varepsilon}\right)^j \frac{\Gamma\left(\frac k2+j\right)}{\Gamma\left(\frac k2\right)} {}_1\F_1\left(\frac k2+j, \frac k2;\frac\lambda2\right).
	\end{split}
\end{equation*}
and (c) is proved. 
\end{proof}

\bibliography{variance}
\bibliographystyle{IEEEtran}

\end{document}